\documentclass[12pt, a4paper, reqno]{amsart}
\usepackage[utf8]{inputenc}
\usepackage{hyperref}
\hypersetup{unicode=true}
\usepackage{xurl}

\usepackage{color}
\usepackage[margin=1.1in,top=1.4in,bottom=1.4in]{geometry}
\usepackage{amsmath, amsthm, amssymb, amsfonts}
\usepackage{calc}
\usepackage{url}
\usepackage{enumitem}
\usepackage{booktabs}

\usepackage{cleveref}
\crefformat{section}{\S#2#1#3} 
\crefformat{subsection}{\S#2#1#3}
\crefformat{subsubsection}{\S#2#1#3}

\numberwithin{equation}{section}
\newtheorem{theorem}{Theorem}[section]

\newtheorem{lemma}[theorem]{Lemma}

\theoremstyle{definition}

\theoremstyle{remark}

\newenvironment{psmallmatrix}
  {\left(\begin{smallmatrix}}
  {\end{smallmatrix}\right)}

\title{Sets whose differences avoid a bracket quadratic}
\author{Khalid Younis}
\address{Mathematics Institute, Zeeman Building, University of Warwick, Coventry CV4
7AL, United Kingdom}
\email{younis.maths@outlook.com}

\begin{document}
\begin{abstract}
Suppose a set of integers \(A\subseteq\{1,\dots,N\}\) has no solutions to \(a-a'=n\lfloor\sqrt[3]2n\rfloor,\) for distinct   \(a,a'\in A,\) and \(n\in \mathbb{N}.\) We show  that \(|A|\ll N^{1-c}\) for some absolute constant \(c>0.\) To do this, we prove quantitative bounds on the van der Corput property for certain sets of bracket quadratics. This comes as a consequence of establishing exponential sum estimates for these sets, utilising a theorem  of Green and Tao on the quantitative equidistribution of polynomial orbits on nilmanifolds, closely following the approach of Neale who went on to prove a Waring-type result. We also extend our result to differences avoiding a  family of bracket polynomials (also known as generalised polynomials).
\end{abstract}

\maketitle

\section{Introduction}

A celebrated theorem of Furstenberg \cite{Furst} and S\'{a}rk\"{o}zy \cite{sarkozy} states that  subsets of integers \(A\subseteq\{1,\dots,N\}\) for which  every non-zero difference \(a-a'\) is not a  square, where \(a,a'\in A,\) must have size \(|A|=o(N)\) as \(N\rightarrow \infty.\)  Furstenberg  established this result using ergodic theory, whereas S\'{a}rk\"{o}zy employed Fourier analysis. The latter approach yielded quantitative bounds of \[|A|\ll N(\log \log N)^{2/3}(\log N)^{-1/3}.\]This was improved by  Pintz, Steiger, and Szemer\'{e}di \cite{pinstesze} to \[|A|\ll N(\log N)^{-(\log \log \log \log N)/12},\] then by Bloom and Maynard \cite{blomay} for a bound of \[[A|\ll N(\log N)^{-c\log \log \log N}\] for some constant \(c>0,\) and most recently by Green and Sawhney \cite{gresaw} to \[|A|\ll Ne^{-c\sqrt{\log N}}.\]  It remains an open problem to prove whether or not \(|A|\ll N^{1-c}.\) Lower bounds are discussed in \cite{gresaw}, for example, though we highlight that  Ruzsa \cite{ruz} constructed such a set of size  \(|A|\gg N^{0.73307\dots},\) which was  improved independently by Beigel and  Gasarch \cite{beigas} and by Lewko \cite{lewko}  to \(|A|\gg N^{0.73341\dots},\) with the current best bound of \(|A|\gg N^{0.75279\dots}\) due to Krachun \cite{krachun}.

There have since been many generalisations, including to differences not equal to polynomials \cite{adajaretal} and multivariate polynomials \cite{doylerice,fanlott}, and to avoiding inhomogeneous progressions \cite{PelPrend}, along with some   accompanying  lower bounds of the author \cite{younis}. See \cite{le,deanetal} for more background and related problems. 
Of particular relevance to us is the problem of avoiding differences equal to shifted primes \(p-1.\)
Forgoing the density increment strategy, Green \cite{greenprime}  achieved a bound of \(|A|\ll N^{1-c}\) by proving a quantitative \emph{van der Corput property} (which we discuss in detail later) for the set of shifted primes. Function field analogues of this result were subsequently proved \cite{fanlott2,kowalska}.

In much earlier work,  Rivat and  S\'{a}rk\"{o}zy  \cite{rivsar} attained a bound of \(|A|\ll N^{1-c}\) for differences avoiding \(\lfloor n^{d} \rfloor\) for suitable \(d\) (see the discussion in \cite[\S1]{okee}).  O'Keeffe \cite{okee} improved this constant \(c\)  for some range of \(d\) by investigating a quantitative van der Corput property, and also extended the result to functions with a similar growth rate to \(\lfloor n^{d} \rfloor.\) Here and throughout, the floor function is defined by \(\lfloor x\rfloor=\max\{y\in \mathbb{Z}: y\leq x\}.\)  Fan and Lott \cite{fanlott} proved a bound of  \(|A|\ll_\varepsilon N^{7/8+\varepsilon}\) for differences avoiding the sum of two squares, again by establishing a quantitative  van der Corput property. Results on the quantitative van der Corput property for squares exist, as in work of Slijep\v{c}evi\'{c}  \cite{slij}, but these are not strong enough to improve known bounds for sets lacking square differences. 

The goal of this article is to show one can achieve similar quantitative bounds for differences avoiding so-called \emph{bracket quadratics.} 
\begin{theorem}\label{mainthm1}
Suppose \(\beta\) and \(\beta^2\) are  real, irrational, and non-Liouville. Let \(N\) be a positive integer, and  let \(A\subseteq\{1,\dots,N\}\) be a set of integers with no solutions to \[a-a'=n\lfloor\beta n\rfloor\] with distinct \(a,a'\in A,\) and \(n\in \mathbb{N}.\) Then \[|A|\ll N^{1-c}\] for some  constant \(c=c(\beta)>0.\)
\end{theorem}
A real  number \(\beta\) is \emph{non-Liouville} when, for some \(m=m(\beta)>0\) and  \(M=M(\beta)\geq 1,\) the inequality 
\[0<\left|\beta-\frac{a}{q}\right|<\frac{m}{q^M}\] has no solutions with integers \(a,q\) such that  \(q\geq 1.\) In other words, \(\beta\) cannot be approximated by different rationals at an arbitrary polynomial rate. Given \(\beta,\) the infimum of such \(M\) is called the \emph{irrationality exponent}. One can verify that all rationals are non-Liouville. A classical theorem of Liouville states that every irrational algebraic number is non-Liouville, and in fact  Roth's theorem shows that any \(M>2\) is permissible  provided that \(m\) also depends on \(M\). For instance, in Theorem \ref{mainthm1} one may take  \(\beta=2^{1/3}\) or \(\beta=\sqrt{2}+1\) (in which case  \(n\lfloor(\sqrt{2}+1) n\rfloor=n\lfloor\sqrt{2} n\rfloor+n^2\)). It is perhaps worth mentioning that the set of Liouville numbers has measure zero, and in fact has Hausdorff dimension zero by Jarn\'{i}k's theorem.

Our result is the first to demonstrate any quantitative bounds on the size of \(A.\) Via ergodic theory, Bergelson and H{\aa}land \cite{berghal}  reached the qualitative conclusion that \(|A|=o(N)\) for a family of bracket polynomials including that in Theorem \ref{mainthm1}.  Related problems involving more general bracket polynomials are handled qualitatively by H\aa land Knutson and MCutcheon \cite{berghalmcc}, and  Bergelson, H\aa land Knutson, and Son \cite{berghaletal}. The ergodic approach is discussed in \cite{le}, for example.  

We have an accompanying lower bound to Theorem \ref{mainthm1} (and Theorem \ref{polythm}  below) of \(|A|\gg N^{1/2}.\)  Indeed,  for a general set \(S\) of non-zero integers,  if \(A\subseteq\{1,\dots , N\}\) is  maximal such that \(a-a'\not\in S\) for all \(a,a'\in A,\) then  maximality implies \(\bigcup_{a\in A} (\{a\} \pm S)\supseteq \{1,\dots, N\},\) and so   by the  union bound  \(2|A||S|\geq \sum_{a\in A}|\{a\}\pm S|\geq |\bigcup_{a\in A} (\{a\} \pm S)|\geq N.\) A construction with a more refined implicit constant is given in \cite[\S{B.2}]{lyall}.

Crucial to our analysis are ideas of Neale \cite{neale} who proved a version of Waring's problem for bracket quadratics \(n\lfloor \sqrt{2} n\rfloor,\) that is, there exists some integer \(s\) such that every sufficiently large number is expressible as the sum of \(s\) of these bracket quadratics. Neale even obtained an asymptotic formula for the count of such solutions. To do this, Neale used the nilpotent Hardy--Littlewood circle method, establishing bounds for exponential sums with a bracket quadratic phase using a theorem of Green and Tao \cite{GT} on the quantitative equidistribution of polynomial sequences on nilmanifolds. One obtains a power saving on the minor arcs, and an explicit formula (up to an error term) on the major arcs.  Wooley \cite{wooltalk} obtained bounds for the moments of exponential sums for bracket quadratics via a variant of Hua's lemma, and also handled moments for other bracket polynomials utilising discrete restriction variants of Vinogradov's mean value theorem; the former significantly reduced the number of variables needed in the Waring-type problem of Neale. 

We also prove a version of our result for a more general collection of bracket polynomials. Here and throughout, \(\mathbb{L}'\subseteq\mathbb{R}\) denotes the set of  non-Liouville numbers. Recall that \(\mathbb{Q}\subseteq\mathbb{L}'\).
\begin{theorem}\label{polythm}
Let \(\ell \in \mathbb{Z}\) and \(Q\in \mathbb{Q}[n]\) be such that \(Q:\mathbb{Z}\rightarrow \mathbb{Z}.\) Suppose \(R,R^2\in \mathbb{L}'[n]\setminus\mathbb{Q}[n].\)  Assume \(\deg Q, \deg R\geq 1\) and  \(Q(0)=R(0)=0.\)  Let \(N\) be a positive integer, and let \(A\subseteq\{1,\dots,N\}\) be a set of integers with no solutions to \[a-a'=Q(n)\lfloor R(n)+\ell \rfloor\] with distinct \(a,a'\in A,\) and \(n\in \mathbb{N}.\) Then \[|A|\ll N^{1-c}\] for some  constant \(c=c(Q,R,\ell)>0.\)

The same result holds with \(R^2\in \mathbb{Q}[n]\) instead of \(R^2\in \mathbb{L}'[n]\setminus\mathbb{Q}[n],\)  provided that \(R\) is not a constant multiple of \(Q.\)
\end{theorem}

\subsection{The van der Corput property}
For sets \(S\subseteq \{-N,\dots, N\}\setminus\{0\}\) (formally sequences of sets for increasing \(N\)) suppose one can construct a cosine polynomial  \[T(\theta)=a_0+\sum_{h\in S}a_h\cos (2 \pi \theta h)\] such that \(a_0,a_h\in \mathbb{R},T(0)=1, T(\theta)\geq 0\) for all \(\theta\in \mathbb{R},\) and \(a_0\rightarrow 0\) as \(N\rightarrow \infty.\) Then \(S\) is said to satisfy the  \emph{van der Corput property}. This is a strictly stronger property than \(S\) being  \emph{intersective}, which means that \(|A|=o(N)\) for any set \(A\subseteq\{1,\dots, N\}\) whose differences avoid \(S.\)  A more thorough discussion of these properties is given in \cite{greenprime, le} and \cite[\S{2}]{mont2}

 Having the van der Corput property implies being intersective. To prove this claim,  consider \[\int_0^1\left|\sum_{a\in A}e(\theta a)\right|^2T(\theta)~\mathrm{d}\theta.\]
 On the one hand, by the hypothesised property on \(A,\) this equals 
 \[a_0\int_0^1\left|\sum_{a\in A}e(\theta a)\right|^2~\mathrm{d}\theta+\Re \int_0^1\left|\sum_{a\in A}e(\theta a)\right|^2\sum_{h\in S}a_h e(\theta h)~\mathrm{d}\theta=a_0|A|.\]
 On the other hand, since the integral is non-negative, we can lower bound it by the contribution near the origin, so that it is at least
 \[\int_{-1/(10N)}^{1/(10N)}\left|\sum_{a\in A}e(\theta a)\right|^2T(\theta)~\mathrm{d}\theta\gg \frac{|A|^2T(0)}{N}= \frac{|A|^2}{N},\] since \(\Re e(\theta n)=\cos(2\pi \theta n)\geq 1/2,\) say, when \(|\theta|\leq 1/(10N)\) and \(|n|\leq N.\)
 Together these imply that \[|A|\ll a_0 N,\] thus proving the claim. Bourgain \cite{bourgrec} proved by a construction  that the reverse implication does not hold in general. 
 
 To prove Theorem \ref{mainthm1}, it therefore suffices to show that such a \(T\) exists for the set of bracket quadratics 
 \[S=\{h=n\lfloor \beta n\rfloor: -N\leq h\leq N, n\in \mathbb{N}\}\setminus\{0\},\]  
  where \(a_0\ll N^{-c}.\) 
We may assume that \(N\) is sufficiently large in terms of \(\beta\) because otherwise Theorem \ref{mainthm1} holds trivially upon adjusting \(c.\) Observe that 
\[S=\{n\lfloor \beta n\rfloor: Y < n\leq X\},\] for some integers \(X=\sqrt{N/|\beta|}+O(1)\) and \(0\leq Y \ll 1/|\beta|\) (which ensure the floor function is non-zero).
 Then \(|S| =X-Y=\sqrt{N/|\beta|}+O(1/|\beta|)+O(1).\)  

 For a suitably small constant \(c>0,\) choose 
 \[a_0=\frac{(N/|\beta|)^{1/2-c}}{(N/|\beta|)^{1/2-c}+|S|} \quad  \textrm{ and } \quad a_h=\frac{1}{(N/|\beta|)^{1/2-c}+|S|} .\] This choice clearly satisfies all the required properties of our cosine polynomial, except possibly \(T(\theta)\geq 0,\)  which will be the main focus of this paper. Following the method of Neale, we prove the following.
 \begin{lemma}\label{nonnegative} Suppose \(\beta\) and \(\beta^2\) are real, irrational, and non-Liouville.  There is  some  constant \(c'=c'(\beta)>0\) such that for all \(\theta\in \mathbb{R}\) and all \(X\)  sufficiently large in terms of \(\beta,\) we have
\[\Re \sum_{1\leq n\leq X}e(-\theta n\lfloor \beta n \rfloor)\geq -X^{1-c'}.\] 
 \end{lemma}
 {This is enough to prove \(T(\theta)\geq 0,\) and hence Theorem \ref{mainthm1}.  It is clear that Theorem \ref{polythm} follows from an analogous result, which we prove in Section \ref{bracketpolysection}. Further details regarding the constants \(c\) and \(c'\) appearing in these results  can be found in Sections \ref{sectionpositive} and \ref{bracketpolysection}. 
 \subsection{Major and minor arcs}
 It turns out that the major arcs for the relevant exponential sum are centred around rationals of small denominator. We show that the exponential sum is small in absolute value (by a power saving), or else \(\theta\approx a/q\)  and \(q\geq 1\) not too large. In the latter case, 
 \[\sum_{1\leq n\leq X}e(-\theta n\lfloor \beta n \rfloor)\approx X\mathbb{E}_{j,m ~(\textrm{mod } q)}e\left(\frac{-ajm}{q}\right).\] The average over residues modulo \(q\) corresponds to the fact that \(n\) and \(\lfloor \beta n\rfloor\) are `independent' modulo \(q.\) Fortunately, this average is easily seen to equal \(1/q\) when  \((a,q)=1,\) and in particular points in the positive real direction. This is contrasted with the case of squares, where we have the Gauss sum \[X\mathbb{E}_{j~(\textrm{mod } q)}e\left(\frac{-aj^2}{q}\right),\]  which may well point in the negative real direction. For example, if \(q=5\) then this average equals \(1/\sqrt{5}\) at \(a=1,\) and \(-1/\sqrt{5}\) at \(a=2.\) As observed by  H{\aa}land  \cite{hal1,hal3} in a qualitative setting, and Neale \cite{neale} in both a qualitative and  quantitative setting, the  major arcs in this problem change dramatically depending on Diophantine properties of \(\beta\) and \(\beta^2,\) and are not always centred around rationals of small denominator as one might expect.  When \(\beta^2\) does not satisfy our hypotheses, for instance if  \(\beta =\sqrt{2},\) one arrives at major arcs at \(\theta\approx(a+b\sqrt{2})/q\) with \(q\geq 1\) such that \(b\) and \(q\) are not too large,  and as Neale \cite{neale} established one instead has
 \begin{equation}\label{root2main}  X \mathbb{E}_{j,m ~(\textrm{mod } 2q)
}e\left(\frac{-ajm}{q}-\frac{b}{2q}(2j^2+m^2)\right)\int_0^1 e\left(\frac{b}{2q}x^2\right) ~\mathrm{d}x.\end{equation}

 For some intuition as to why quadratic irrationals such as \(\beta=\sqrt{2}\) present such different behaviour,  observe that
\begin{align*}\{\sqrt 2n\}^2 &=\left(\sqrt 2 n-\lfloor \sqrt{2}n \rfloor \right)^2\\
&=2n^2+\lfloor \sqrt{2}n \rfloor^2-2\sqrt{2}n\lfloor \sqrt{2}n \rfloor \\
&\equiv - 2\sqrt{2} n \lfloor \sqrt{2} n \rfloor \pmod 1.
\end{align*}
According to H\aa land \cite{hal1,hal3}, this observation is due to  Ruzsa. It appears to have been independently noted in comments by Tao \cite{taoblog}.
 If we believe that \(\left(\{\sqrt 2 n\}\right)_{n\in \mathbb{N}}\) is equidistributed in \(\mathbb{R}/\mathbb{Z},\) then  \(\left(\{\sqrt{2} n\}^2\right)_{n\in \mathbb{N}}\) is not. 
  In particular, \(2\sqrt 2 n\) and \(\lfloor \sqrt 2 n\rfloor\)  do not behave `independently'.  For \(\beta\in \mathbb{R}\setminus\mathbb{Q},\) H\aa land \cite{hal1} established that \((\theta n \lfloor \beta n \rfloor \mod 1)_{n\in \mathbb{N}}\)  is equidistributed in \(\mathbb{R}/\mathbb{Z}\) if and only if either \(\beta^2\in \mathbb{Q}\) and \(1,\theta,\beta\) are linearly independent over \(\mathbb{Z},\) or \(\beta^2\not \in \mathbb{Q}\) and \(\theta\not \in \mathbb{Q}\) (see also the proof in \cite[\S 6]{neale} and the discussion in \cite[\S1, \S5]{hal3}). As well as works previously mentioned, qualitative equidistribution results for other bracket polynomials (also called generalised polynomials) include \cite{veech,peres,hal3,hal1,Hal2,halknut,bergleib}.
  
 We closely follow the  nilpotent approach of Neale for evaluating exponential sums.  In addition to the natural changes to the analysis of horizontal characters due to a different Diophantine condition on \(\beta^2\) (see Sections \ref{sectiong2} and \ref{sectiong3}),  we  cover major and minor arc estimates for   exponential sums of more general bracket polynomials  coming from the Heisenberg nilmanifold (see Section \ref{bracketpolysection}),  which presents further technical challenges. We remark that a somewhat simpler quantitative equidistribution result for a similar  family of  bracket polynomials to ours is also handled in \cite[\S 6]{pand}, and that the result corresponding to \eqref{root2main} for general \(\beta\in \mathbb{R}\setminus \mathbb{Q}\) such that \(\beta^2 \in \mathbb{Q}\) is presented in \cite{dask} (as well as Lemma \ref{majorarcrootestimate} below). Further differences between our treatment and that of Neale include our introduction of Sobolev theory to  estimate Lipschitz norms (see Section \ref{nilbackground}) and our means to avoid the need for `conjugate nilmanifolds'  (see Section \ref{sectiong2}).   A more detailed introduction to our problem  is presented by the author in \cite[\S 1]{ younisthesis}.

\subsection*{Notation} 
We adopt standard asymptotic notation. We write \(f\ll g\) or \(f=O(g)\) to mean \(|f|\leq C g\) for some constant \(C>0.\) Let \(f\asymp g\) signify \(f\ll g \ll f.\) When the implicit constant depends on some parameter, this is indicated by a subscript.  For some parameter tending to infinity (which should be clear from context) we also write \(f=o(g)\) if \(f/g\rightarrow 0,\) and \(f\sim g\) if \(f/g\rightarrow 1.\)

  We write \(e(\theta)=e^{2\pi i \theta}.\)   Given a set or event \(E,\) we write \(1_E\) for its indicator function.  The notation \(\Vert x \Vert_{\mathbb{R}/\mathbb{Z}}\) denotes the distance from \(x\) to the nearest integer. The floor function is defined by \(\lfloor x\rfloor=\max\{y\in \mathbb{Z}: y\leq x\}\) and the ceiling function by \(\lceil x\rceil=\min\{y\in \mathbb{Z}: y\geq x\}.\) For the greatest common divisor of \(a\) and \(b\) we write \((a,b).\)  For sets \(A\) and \(B,\) the sumset is denoted \(A+B=\{a+b: a\in A, \ b\in B\}.\) 
We also use the shorthand \([N]=\{1\leq n\leq N: n\in \mathbb{Z}\}.\) For a non-empty set \(S,\)  the notation \(\mathbb{E}_{n\in S}\) means the normalised sum \(|S|^{-1}\sum_{n\in S},\) and we may abbreviate this to \(\mathbb{E}_n\) when the set \(S\) is clear from context. We also write \(S[n]\) for the set of polynomials in \(n\) whose coefficients are in \(S,\) even when \(S\) is not a ring. 
   
 \subsection*{Acknowledgements}
 The author is thankful for the  guidance of Sam Chow. The author is also grateful to Vicky Neale for sharing her thesis and for her encouragement, to Joel Moreira for pointing out some relevant papers in the ergodic theory literature, to Joni Ter\"{a}v\"{a}inen for helpful comments, and to Sean Eberhard for an enlightening talk and discussions on the algebraic theory of nilsequences. 
The author was supported by the Warwick Mathematics Institute Centre for Doctoral Training, and gratefully acknowledges funding by the Swinnerton-Dyer scholarship. 
\subsection*{Rights Retention} For the purpose of open access, the author has applied a Creative Commons Attribution (CC-BY) licence to any Author Accepted Manuscript version arising
from this submission.
 \subsection{Organisation} Section \ref{nilbackground} covers background theory on polynomial equidistribution on nilmanifolds, as well as some Sobolev theory to facilitate the estimation  of Lipschitz norms which we invoke throughout the argument. In Section \ref{sectiong1} we set up the problem in terms of a polynomial sequence on the Heisenberg  nilmanifold. We show that a lack of equidistribution leads to \(\theta\) being somewhat close to \((a+b\beta)/q\) for integers \(a,b,\) and \(q\) which are not too large. For such \(\theta,\)  we descend onto a subnilmanifold in Section \ref{sectiong2} and show that in fact \(b=0\) and that \(\theta\) is even closer to rational numbers \(a/q\) with \(a\) and \(q\) not too large. In Section \ref{sectiong3} we pass to yet another subnilmanifold to establish a major arc estimate for these \(\theta.\) We show in Section \ref{sectionpositive} that the main term points in the positive real direction to establish Lemma \ref{nonnegative}, thus proving Theorem \ref{mainthm1}. The goal of Section \ref{bracketpolysection} is to adapt the arguments to handle more general bracket polynomials and  prove Theorem \ref{polythm}.

 \section{Background theory}\label{nilbackground}
 We begin this section by covering some background material on group theory and equidistribution, much of which can be found in \cite{GT}.
 Let \(G\) be a connected, simply connected nilpotent Lie Group. A \emph{filtration} \(G_\bullet\) of \emph{degree} \(d\) is a sequence of closed connected subgroups
\[G=G_0=G_1\geq G_2 \geq \dots \geq G_d \geq G_{d+1}=\{1\}\] with the commutators satisfying  \([G_i,G_j]\leq G_{i+j}.\) (The notation \(G_i\) is not to be confused with the same notation used for specific groups defined  throughout the paper.)

 Let \(\mathfrak{g}\) be the Lie algebra associated to \(G.\) Suppose \((X_1,\dots,X_m)\) is a basis of \(\mathfrak{g}.\) The structure constants \(c_{i,j,k}\) are given by the relations of Lie brackets \[[X_i,X_j]=\sum_{k}c_{i,j,k}X_k.\] Mal'cev's criterion states that the structure constants are rational if and only if \(G\) admits a lattice \(\Gamma\) (meaning a discrete co-compact subgroup). For \(Q>0,\) the basis is said to be \emph{\(Q\)-rational} if the structure constants are rationals of height at most \(Q.\) A \emph{nilmanifold} is the quotient space \(G/\Gamma=\{g\Gamma:g\in G\}.\) The \emph{dimension} of \(G/\Gamma\) is the dimension of \(\mathfrak{g},\) equal to \(m.\)
 
 A \emph{Mal'cev basis} for \(G/\Gamma\) is a basis \((X_1,\dots,X_m)\) of \(\mathfrak{g}\) such that:
 \begin{itemize}
 \item The subspace \(\mathfrak{h}_j=\langle X_{j},\dots, X_m\rangle\) is a Lie algebra ideal of \(\mathfrak{g}\) for  \(j=1,\dots,m+1.\) 
 \item Each \(g\in G\) can be written uniquely as \(g=e^{t_1X_1}\dots e^{t_mX_m}\) for \(t_i\in \mathbb{R},\) and \(g\in \Gamma\) if and only if all \(t_i\in \mathbb{Z}.\)
 \end{itemize}
 In this context, \(e\) with an exponent denotes the exponential map from \(\mathfrak{g}\) to \(G.\) (From Section \ref{sectiong1} onwards, we implicitly make use of the formula \(e^X=\sum_{i\geq 0}X^i/i!,\) with \(X\) a square matrix.) One can show that every  \(g\Gamma\in G/\Gamma\) is uniquely expressible as \(g\Gamma=e^{t_1X_1}\dots e^{t_mX_m}\Gamma\) with all \(t_i\in [0,1),\) and so this corresponds to a \emph{fundamental domain} of \(G/\Gamma.\)
 Writing \(H_j=e^{\mathfrak{h}_j}\) and \(H_0=H_1,\) one may show that  
 \[G=H_0=H_1\geq H_2\geq \dots \geq H_{m+1}=\{1\}\] is a filtration. The Mal'cev basis is \emph{adapted to \(G_\bullet\)} if each \(G_j\) is equal to some \(H_i.\)
 
There are several equivalent definitions of a polynomial sequence \(g:\mathbb{Z}\rightarrow G.\) One in terms of a Mal'cev basis is that \(g(n)=e^{t_1(n)X_1}\dots e^{t_m(n)X_m}\) for some ordinary polynomials \(t_i:\mathbb{Z}\rightarrow\mathbb{R}.\) We say that \(g\) is a polynomial \emph{with respect to  \(G_\bullet\)} if  \(\deg t_i <j\) for  \(1\leq i\leq m\) whenever \(G_j\leq H_{i+1}.\) 
Given a filtration \(G_\bullet\) of degree \(d\) adapted to a Mal'cev basis, \(g\) is  always a polynomial with respect to a new filtration adapted to the Mal'cev basis, constructed by inserting sufficiently many copies of \(G\) into the filtration \(G_\bullet.\) The degree of the new filtration is then bounded in terms of \(d\) and \(\max_i \deg t_i.\)

Central to our analysis is the following   quantitative version of Leibman's theorem  by Green and Tao \cite{GT} (the proof of this statement, with the slightly stronger hypothesis of being not \emph{totally} equidistributed, is deduced in \cite[Cor.~2.24]{neale}, or in \cite[Thm.~3.5]{shaoter}).
\begin{theorem}[Green--Tao quantitative equidistribution]\label{GTequid}
Let \(m,d\geq 0,0<\delta<1/2,\) and \(X\geq 1.\) Suppose that \(G/\Gamma\) is a nilmanifold of dimension \(m,\) with a filtration \(G_\bullet\) of degree \(d,\) and  a \(\delta^{-1}\)-rational Mal'cev basis adapted to \(G_\bullet.\)   Let \(g:\mathbb{Z}\rightarrow G\) be a polynomial with respect to \(G_\bullet.\) If \((g(n)\Gamma)_{n\in [X]}\) is not totally \(\delta\)-equidistributed in \(G/\Gamma,\) then there exists a horizontal character \(\eta:G\rightarrow\mathbb{R}/\mathbb{Z}\) such that \(0<|\eta|\ll\delta^{-O_{d,m}(1)}\) and 
\[\lVert \eta \circ g \rVert_{C^\infty[X]}\ll\delta^{-O_{d,m}(1)}.\]
\end{theorem}
Let us define some of the terminology appearing in this theorem.  For an ordinary  polynomial \(\alpha(n)=\alpha_0+\alpha_1\binom{n}{1}+\dots +\alpha_s\binom{n}{s},\) the \emph{smoothness norm} is defined by \[\Vert \alpha\Vert_{C^\infty[X]}=\sup_{1\leq j\leq s}X^j\Vert\alpha_j\Vert_{\mathbb{R}/\mathbb{Z}}. \]

 A sequence \((g(n)\Gamma)_{n\in[X]}\) in \(G/\Gamma\) is said to be  \emph{\(\delta\)-equidistributed} in \(G/\Gamma\) when 
\[\left|\mathbb{E}_{n\in [X]}F(g(n)\Gamma)-\int_{G/\Gamma}F\right|\leq \delta\lVert F \rVert _{\textup{Lip}}\] for all Lipschitz functions \(F:G/\Gamma\rightarrow \mathbb{C}.\) We say \((g(n)\Gamma)_{n\in[X]}\) is \emph{totally \(\delta\)-equidistributed} in \(G/\Gamma\) when
\[\left|\mathbb{E}_{n\in P}F(g(n)\Gamma)-\int_{G/\Gamma}F\right|\leq \delta\lVert F \rVert _{\textup{Lip}}\]
for all Lipschitz functions \(F:G/\Gamma\rightarrow \mathbb{C}\)  and for all arithmetic progressions \(P\subseteq\mathbb[X]\) of size \(|P|\geq \delta X.\) Clearly, being totally \(\delta\)-equidistributed implies being \(\delta\)-equidistributed. 

The Lipschitz constant for a continuous function \(F:G/\Gamma\rightarrow \mathbb{C}\) is defined by
\[\lVert F \rVert _{\textup{Lip}}=\Vert F \Vert _{\infty}+\sup_{\substack{g\Gamma,h\Gamma\in G/\Gamma\\g\Gamma \neq h\Gamma}}\frac{|F(g\Gamma)-F(h\Gamma)|}{d(g\Gamma,h\Gamma)},\] for a suitable  distance function \(d\) which we discuss later.  Although we do not use the terminology elsewhere, a \emph{nilsequence} is a function of the form \(n\mapsto F(g(n)\Gamma)\) with  \(g:\mathbb{Z}\rightarrow G\)  a polynomial sequence and \(F:G/\Gamma\rightarrow\mathbb{C}\)  Lipschitz. However, we also need to work with functions which are not necessarily continuous. 

Given a nilmanifold \(G/\Gamma,\) a \emph{horizontal character} \(\eta:G\rightarrow \mathbb{R}/\mathbb{Z}\) is a continuous additive homomorphism, so that \(\eta(gh)=\eta(g)+\eta(h),\) such that \(\eta(\Gamma)=\{0\}.\) It turns out that every horizontal character is of the form \(\eta(g)=(k_1,\dots k_m)\cdot(t_1,\dots,t_m),\) where \(k_i\in \mathbb{Z}\) and  \(g=e^{t_1X_1}\dots e^{t_mX_m}\) is written in terms of a Mal'cev basis. The \emph{modulus} of \(\eta\) is defined as \(|\eta|=\Vert(k_1,\dots,k_m)\Vert_2=\sqrt{k_1^2+\dots+ k_m^2}.\) 

Given a  basis \((X_1,\dots,X_m)\) for a nilmanifold \(G/\Gamma,\) we can define an inner product on \(\mathfrak{g}\) (i.e.~a Riemannian metric) by 
\[\left\langle \sum_ia_iX_i,\sum_ib_iX_i\right\rangle=\sum_ia_ib_i.\] This induces a (Euclidean) distance function \(d_2:G/\Gamma\times G/\Gamma\rightarrow [0,\infty)\) given by 
\[d_2(g\Gamma,h\Gamma)=\inf\left\{\int_0^1\Vert \gamma'(t) \Vert~\mathrm{d}t\right\},\] where the infimum is taken over all piecewise smooth curves \(\gamma:[0,1]\rightarrow G/\Gamma\) such that \(\gamma(0)=g\Gamma\) and \(\gamma(1)=h\Gamma.\) One may show that \(d_\infty(g\Gamma,h\Gamma) \asymp_m d_2(g\Gamma,h\Gamma),\) where \(d_\infty\) is the distance function introduced in \cite{GT}. Since we view \(m\) as bounded, we may instead  calculate Lipschitz constants with respect to \(d_2,\) and work instead with the norm associated to the  Sobolev space \(W^{1,\infty}(G/\Gamma).\) The idea of using Sobolev norms is not new (see e.g.~\cite{bloominverse}), but it seems not to have been put into practice before.   Some background  on Sobolev spaces on Riemann manifolds  can be found in \cite{heb}, for example.
  With respect to Mal'cev coordinates  we have the following:   writing \(F_{h\Gamma}(g)=F(gh\Gamma)\) for \(g,h\in G,\)  we have \[\sup_{\substack{g\Gamma,h\Gamma\in G/\Gamma\\g\Gamma \neq h\Gamma}}\frac{|F(g\Gamma)-F(h\Gamma)|}{d_2(g\Gamma,h\Gamma)}=\textrm{ess}\sup_{h\in G}\Vert \nabla F_{h\Gamma}|_0 \Vert_2=\textrm{ess}\sup_{h\in G}\sqrt{\sum_{1\leq i\leq m} \left|\frac{\partial F_{h\Gamma}}{\partial t_i}\rvert_0\right|^2 }.\] Here the essential supremum ensures this makes sense even for Lipschitz functions which are not differentiable everywhere (which should be compared with Rademacher's theorem from classical analysis).  As it happens, in our work  we only consider such calculations with  smooth functions, in which case  the essential supremum may be replaced by an ordinary supremum.

  For particular functions \(F:G/\Gamma\rightarrow \mathbb{C}\) with discontinuities, we work with smooth functions \(\Psi:G/\Gamma\rightarrow [0,1]\) such that \(\Psi F\) is smooth (and in particular Lipschitz). We take smooth \(\Psi\) to equal \(0\) around the points of discontinuity of \(F,\) to equal \(1\) outside some small neighbourhood of such points, and to smoothly interpolate the remaining values. 
The purpose of this transition into working with Sobolev norms is that it makes calculating certain Lipschitz norms a simple matter of differentiation and using the product rule  \(\nabla f_1f_2=f_1\nabla f_2  + f_2 \nabla f_1.\) In contrast,  such calculations can be difficult  with the norm provided in \cite{GT}.
 
\section{Equidistribution in \texorpdfstring{$G_1/\Gamma_1$}{G1/Γ1}}\label{sectiong1}
 Let \(G_1=\left\{\begin{psmallmatrix}
1 & x & z\\
0 & 1 & y\\
0 & 0 & 1
\end{psmallmatrix}: x,y,z\in \mathbb{R}\right\}\) and  \(\Gamma_1=\left\{\begin{psmallmatrix}
1 & x & z\\
0 & 1 & y\\
0 & 0 & 1
\end{psmallmatrix}: x,y,z\in \mathbb{Z}\right\},\) with Mal'cev basis given by
\(X_1=\begin{psmallmatrix}
0 & 1 & 0\\
0 & 0 & 0\\
0 & 0 & 0
\end{psmallmatrix}, X_2=\begin{psmallmatrix}
0 & 0 & 0\\
0 & 0 & 1\\
0 & 0 & 0
\end{psmallmatrix}\) and \(X_3=\begin{psmallmatrix}
0 & 0 & 1\\
0 & 0 & 0\\
0 & 0 & 0
\end{psmallmatrix}.\)   Let \(g_1(n)=\begin{psmallmatrix}
1 & \theta n & 0\\
0 & 1 & \beta  n\\
0 & 0 & 1
\end{psmallmatrix},\) with \(\theta\) real and with \(\beta\) and \(\beta^2\) both real, irrational, and non-Liouville. 
 One may easily verify that our Mal'cev basis is \(1\)-rational, and that \(g_1(n)\) is a polynomial sequence in \(G_1.\) Note that  \[
\begin{pmatrix}
1 & x & z\\
0 & 1 & y\\
0 & 0 & 1
\end{pmatrix}\begin{pmatrix}
1 & x' & z'\\
0 & 1 & y'\\
0 & 0 & 1
\end{pmatrix}=\begin{pmatrix}
1 & x+x' & z+z'+xy'\\
0 & 1 & y+y'\\
0 & 0 & 1
\end{pmatrix},\] where the \(xy'\) highlights that \(G_1\) is not Abelian.  Observe that
\begin{equation}\label{fund}\begin{pmatrix}1&x&z\\ 0&1&y\\ 0&0&1\\
\end{pmatrix}\begin{pmatrix}1&-\lfloor{x}\rfloor &-\lfloor{z-x\lfloor{y}\rfloor}\rfloor\\ 0&1&-\lfloor{y}\rfloor\\ 0&0&1\\
\end{pmatrix}=\begin{pmatrix}1&\{x\}&\{z-x\lfloor{y}\rfloor\}\\ 0&1&\{y\}\\ 0&0&1\\
\end{pmatrix},\end{equation} so every \(g_1\Gamma_1\in G_1/\Gamma_1\) has a representative with all three coordinates in \([0,1).\) A simple matrix calculation reveals that this representative is unique. Thus, we may view \(G_1/\Gamma_1,\) known as the Heisenberg nilmanifold, as a twisted cube \([0,1]^3.\) 

\begin{lemma}[Equidistribution in \(G_1/\Gamma_1\) implies cancellation]\label{equidG1} Let \(0<\delta_1<1/10.\) If \((g_1(n)\Gamma_1)_{n\in [X]}\) is \(\delta_{1}\)-equidistributed in \(G_1/\Gamma_1,\) then \[\mathbb{E}_{n\in [X]}e(-\theta n \lfloor \beta  n\rfloor)\ll \delta_1^{1/2}.\]  
\end{lemma}
\begin{proof}
Let \(F_1:G/\Gamma_1\rightarrow \mathbb{C}\) be defined by  \(F_1\left(\begin{psmallmatrix}
1 & x & z\\
0 & 1 & y\\
0 & 0 & 1
\end{psmallmatrix}\Gamma_1\right)=e(z-x\lfloor y\rfloor).\) This is well-defined (consider the upper-right coordinate in \eqref{fund}), but is not continuous (when \(y\in \mathbb{Z}\) and \(x\not \in \mathbb{Z}\)), so we cannot use it as a test function for equidistribution. Let us introduce a smooth weight 
\(\Psi_1\left(\begin{psmallmatrix}
1 & x & z\\
0 & 1 & y\\
0 & 0 & 1
\end{psmallmatrix}\Gamma_1\right)=\psi(y),\) where \(\psi=\psi_\varepsilon:\mathbb{R}\rightarrow [0,1]\) is a \(1\)-periodic  function which, for \(-1/2\leq y<1/2,\) equals 0 when \(|y|<\varepsilon,\) equals \(1\) when \(|y|>2\varepsilon,\) and smoothly interpolates the rest. Here \(\varepsilon>0\) is a small quantity to be chosen later. Now the product \(\Psi_1 F_1\) is continuous.

By the equidistribution property, 
\[\left|\mathbb{E}_{n\in [X]} \Psi _1F_{1}(g_1(n)\Gamma_1)-\int_{G_1/\Gamma_1}  \Psi_1 F_{1}\right|\leq \delta_1 \Vert \Psi_1 F_1 \Vert_{\textrm{Lip}}.\] 
Let us address each term involving \(\Psi_1 F_{1}.\) In terms of Mal'cev coordinates, we have \(F_1\left(e^{t_1X_1}e^{t_2X_2}e^{t_3X_3}\Gamma_1\right)=e(t_3+t_1t_2-t_1\lfloor t_2\rfloor)\) and \(\Psi_1\left(e^{t_1X_1}e^{t_2X_2}e^{t_3X_3}\Gamma_1\right)=\psi(t_2),\) so that \[\int_{G_1/\Gamma_1} \Psi_1 F_1=\int_0^1  \int_0^1  \int_0^1  e(t_3+t_1t_2-t_1\lfloor t_2\rfloor)\psi(t_2)
~\mathrm{d}t_1\, \mathrm{d}t_2\,\mathrm{d}t_3=0,\] since the integral over \(t_3\) is \(0.\)

Next, using that  \(\Vert F_1 \Vert_\infty\leq 1,\) and that \(1-\Psi_1\) is continuous with \( \Vert 1-\Psi_1  \Vert_{\textrm{Lip}}\ll \varepsilon^{-1},\) we estimate  the difference 
\begin{equation}\label{smoothdiffexp}
\begin{split}
\left|\mathbb{E}_{n\in [X]}F_1(g_1(n)\Gamma_1)-\mathbb{E}_{n\in [X]}\Psi_1 F_{1}(g_1(n)\Gamma_1)\right|&\leq   \mathbb{E}_{n\in [X]}(1-\Psi_1 )(g_1(n)\Gamma_1)\\&=\int_{G_1/\Gamma_1} (1-\Psi_1)+O(\delta_1 \Vert 1-\Psi_1 \Vert_{\textrm{Lip}})\\&=\int_0^1 1- \psi(t_2)~\mathrm{d}t_2 +O(\delta_1 \varepsilon^{-1})\\&\ll \varepsilon + \delta_1 \varepsilon^{-1}.
\end{split}
\end{equation}

 Finally, one may verify that \(\Vert \Psi_1 F_1 \Vert_{\textrm{Lip}}\ll \varepsilon^{-1}\) following the discussion at the end of Section \ref{nilbackground}. 
Take \(\varepsilon=\delta_1^{1/2}\) to complete the proof.
\end{proof}
From now unitl Section \ref{bracketpolysection}, let \(X\) be sufficiently large in terms of \(\beta,\)  let \(\kappa_1,\kappa_2,\kappa_3>0\) be   sufficiently small.  Let \(\mu\geq 1\) be an uppoer bound for the irrationality exponents of both \(\beta\) and \(\beta^2.\)   Let \begin{equation}\label{deltas}\delta_1=X^{-\kappa_1/\mu},\delta_2=X^{-\kappa_2/\mu}, \textrm{ and } \delta_3=X^{-\kappa_3/\mu}.\end{equation} We also insist that \(\kappa_1\) is sufficiently small in terms of \(\kappa_2,\) and that \(\kappa_2\) is sufficiently small in terms of \(\kappa_3.\)
\begin{lemma}[Lack of equidistribution in \(G_1/\Gamma_1\) implies being close to major arcs]\label{delta1equid} Let  \(\delta_1\) satisfy \eqref{deltas} and let \(X\) be sufficiently large in terms of \(\beta.\) If \((g_1(n)\Gamma_1)_{n\in [X]}\) is not \(\delta_1\)-equidistributed in \(G_1/\Gamma_1,\) then \begin{equation}\label{theta1}\theta=\frac{a+b\beta }{q}+t\end{equation} for integers \(a,b,q\) with \(b,q\ll \delta_1^{-O(1)}, q\geq 1,\) and \(t\ll \delta_1^{-O(1)}X^{-1}.\) 
\end{lemma}
\begin{proof}
By Theorem \ref{GTequid} there exists a non-trivial horizontal character \(\eta:G_1\rightarrow \mathbb{R}/\mathbb{Z},\) with \(|\eta|\ll \delta_1^{-O(1)},\) such that \(\Vert\eta \circ g_1\Vert_{C^\infty[X]}\ll \delta_1^{-O(1)}.\) By the additive property of \(\eta,\) one finds that such characters are of the form \(\eta\begin{psmallmatrix}
1 & x & z\\
0 & 1 & y\\
0 & 0 & 1
\end{psmallmatrix}=k_1x+k_2y\) for integers \(k_1,k_2\ll\delta_1^{-O(1)} \) not both zero. Since \(\eta(g_1(n))=k_1\theta n+ k_2 \beta  n,\) we see that 
\begin{equation}\label{linrelation} \Vert k_1\theta + k_2 \beta  \Vert_{\mathbb{R}/\mathbb{Z}}\ll \frac{\delta_1^{-O(1)}}{X}.\end{equation}  

Suppose for contradiction that \(k_1=0,\) and so \(k_2\neq 0.\) Since \(\beta\) is irrational and non-Liouville, there exists a positive integer \(M=M(\beta)\leq 2\mu\) such that \[\Vert k_2 \beta  \Vert_{\mathbb{R}/\mathbb{Z}}\gg_\beta \frac{1}{|k_2|^{M-1}} \gg \delta_1^{O(M)}.\]  Therefore \(X \ll \delta_1^{-O(M)},\) which is impossible for large \(X\) given \eqref{deltas}. Thus \(k_1\neq 0\) and the result follows from \eqref{linrelation}.
\end{proof}
\section{Equidistribution in \texorpdfstring{$G_2/\Gamma_2$}{G2/Γ2}}\label{sectiong2}
Let \(G_2=\left\{\begin{psmallmatrix}
1 & bx & y\\
0 & 1 & qx\\
0 & 0 & 1
\end{psmallmatrix}: x,y\in \mathbb{R}\right\}\) and  \(\Gamma_2=\left\{\begin{psmallmatrix}
1 & bx & y\\
0 & 1 & qx\\
0 & 0 & 1
\end{psmallmatrix}: x,y-bqx^2/2\in \mathbb{Z}\right\},\)
with Mal'cev basis given by
\(Y_1=\begin{psmallmatrix}
0 & b & 0\\
0 & 0 & q\\
0 & 0 & 0
\end{psmallmatrix}\) and \(Y_2=\begin{psmallmatrix}
0 & 0 & 1\\
0 & 0 & 0\\
0 & 0 & 0
\end{psmallmatrix}.\) This basis is certainly \(1\)-rational. In fact, all the structure constants equal \(0\) because \(Y_1\) and \(Y_2\) commute with one another.
 When \(\theta\) is as in \eqref{theta1},
 we can factorise the polynomial
\begin{align*}
g_1(n)&=\begin{pmatrix}
1 & \theta n & 0\\
0 & 1 & \beta n\\
0 & 0 & 1
\end{pmatrix}\\&=\begin{pmatrix}
1 & \left(\frac{a+b\beta }{q}+t\right) n & 0\\
0 & 1 & \beta n\\
0 & 0 & 1
\end{pmatrix}
\\&=\begin{pmatrix}
1 & t n & 0\\
0 & 1 & 0\\
0 & 0 & 1
\end{pmatrix}\begin{pmatrix}
1 & \frac{b\beta  n}{q} & -\beta tn^2\\
0 & 1 & \beta n\\
0 & 0 & 1
\end{pmatrix}
\begin{pmatrix}
1 & \frac{an}{q} & 0\\
0 & 1 & 0\\
0 & 0 & 1
\end{pmatrix}
\\&=\varepsilon_2(n)g_2(n)\gamma_2(n).\end{align*}
Here \(\varepsilon_2(n)\) is `smooth' as it is approximately constant because \(t\) is small, \(g_2(n)\in G_2\) is a polynomial, and \(\gamma_2(n)\Gamma_1\) is `periodic' as it depends only on \(n\) modulo \(q.\)  (A more general factorisation framework of this kind  is developed by Green and Tao in \cite{GT}.) 

\begin{lemma}[Total equidistribution in \(G_2/\Gamma_2\) implies cancellation]\label{totalequidg2} Let \(\delta_1,\delta_2\) satisfy \eqref{deltas} and let \(X\) be sufficiently large in terms of \(\beta.\) Suppose \(\theta\) is of the form \eqref{theta1}. If \((g_2(n)\Gamma_2)_{n\in [X]}\) is totally \(\delta_2\)-equidistributed in \(G_2/\Gamma_2,\) then
\[\mathbb{E}_{n\in [X]}e(-\theta n \lfloor \beta  n\rfloor)\ll \delta_1^{-O(1)}\delta_2^{1/2}.\] 
\end{lemma}
\begin{proof} Recall that \(F_1(g_1(n)\Gamma_1)=e(-\theta n \lfloor \beta n\rfloor)\) and  \(\Psi_1\left(\begin{psmallmatrix}
1 & x & z\\
0 & 1 & y\\
0 & 0 & 1
\end{psmallmatrix}\Gamma_1\right)=\psi(y).\) 
 By total equidistribution, \[\left|\mathbb{E}_{n\in P} \Psi_2 F_2(g_2(n)\Gamma_2)-\int_{G_2/\Gamma_2} \Psi_2 F_2\right|\leq \delta_2\Vert \Psi_2 F_2 \Vert_{\textrm{Lip}} \] for all arithmetic progressions \(P\subseteq [X]\) of size \(|P|\geq \delta_2X,\) where \(\Psi_2,F_2:G_2/\Gamma_2\rightarrow\mathbb{C}\) are given by \(F_2(g_2\Gamma_2)=F_1(\varepsilon_2g_2\gamma_2\Gamma_1)\) for some specially chosen \(\varepsilon_2,\gamma_2\in G_2,\)  and \(\Psi_2(g_2\Gamma_2)=\Psi_1(\varepsilon_2g_2\gamma_2\Gamma_1)\) with \(\varepsilon=\delta_2^{1/2}.\) We emphasise that \(F_2\) and \(\Psi_2\) depend on the constants \(\varepsilon_2\) and \(\gamma_2.\)  

 The functions \(\Psi_2\) and \(F_2\) are  well-defined  provided \(\gamma_2^{-1}\Gamma_2\gamma_2\subseteq \Gamma_1.\) Indeed, if \(g_2\Gamma_2=g_2^*\Gamma_2\) then \(g_2\gamma_2\gamma_2^{-1}\Gamma_2\gamma_2=g_2^*\gamma_2\gamma_2^{-1}\Gamma_2\gamma_2,\) so \(g_2\gamma_2\Gamma_1=g_2^*\gamma_2\Gamma_1,\) and thus \(\varepsilon_2 g_2\gamma_2\Gamma_1=\varepsilon_2 g_2^*\gamma_2\Gamma_1.\)
 By writing in terms of Mal'cev coordinates (see below), we have \(\int_{G_2/\Gamma_2}(1-\Psi_2)=\int_0^11-\psi(qt_1)\mathrm{d}t_1\ll \delta_2^{1/2},\) and we also find that  \(\Vert \Psi_2 \Vert_{\textup{Lip}},\Vert \Psi_2 F_2\Vert_{\textup{Lip}}\ll \delta_1^{-O(1)}\delta_2^{-1/2}\)  by the discussion at the end of Section \ref{nilbackground}.

Partition the interval \([X]\) into arithmetic progressions \(P_{j,k}\) of common difference \(q,\) with elements congruent to \(j\) modulo \(q,\) and with size \(\delta_2 X\leq |P_{j,k}|\leq 2\delta_2 X.\) 
Pick an element \(n_{j,k}\in P_{j,k}.\) Then
\begin{align*}\mathbb{E}_{n\in [X]}F_1(g_1(n)\Gamma_1)
&=\mathbb{E}_{n\in[X]}F_1(\varepsilon_2(n)g_2(n)\gamma_2(n)\Gamma_1)\\ &\leq \max_{j,k}\left|\mathbb{E}_{n\in P_{j,k}}F_1(\varepsilon_2(n)g_2(n)\gamma_2(j)\Gamma_1)\right|\\&= \max_{j,k}\left|\mathbb{E}_{n\in P_{j,k}}F_1(\varepsilon_2(n_{j,k})g_2(n)\gamma_2(j)\Gamma_1)\right|+E\\
&= \max_{j,k}\left|\mathbb{E}_{n\in P_{j,k}}F_2(g_2(n)\Gamma_2)\right|+E
\\&=\max_{j,k}\left|\mathbb{E}_{n\in P_{j,k}}\Psi_2F_2(g_2(n)\Gamma_2)\right|+O(\delta_1^{-O(1)}\delta_2^{1/2})+E
\\&=\left|\int_{G_2/\Gamma_2} \Psi_2 F_2\right|+O(\delta_1^{-O(1)}\delta_2^{1/2})+E.
\end{align*} Here the first \(O(\delta_1^{-O(1)}\delta_2^{1/2})\) term comes from a nearly identical calculation to \eqref{smoothdiffexp}, and 
\[E=\max_{j,k}\left|\mathbb{E}_{n\in P_{j,k}}F_1(\varepsilon_2(n)g_2(n)\gamma_2(j)\Gamma_1)-F_1(\varepsilon_2(n_{j,k})g_2(n)\gamma_2(j)\Gamma_1)\right|\] is the error incurred by replacing \(n\) with a representative element \(n_{j,k}\) in the `smooth' part, and where we choose \(\varepsilon_2=\varepsilon_2(n_{j,k})\) and \(\gamma_2=\gamma_2(j)\) for the definition of \(\Psi_2\) and \(F_2\) (and although \(\Psi_2\) and \(F_2\) depend on \(j\) and \(k\) for example, we suppress this in our notation).  The fact that this \(\gamma_2\) satisfies the property \(\gamma_2^{-1}\Gamma_2\gamma_2\subseteq \Gamma_1\) follows from the calculation
\begin{equation}\label{conjugation}
\begin{pmatrix}
1 & \frac{-aj}{q} & 0\\
0 & 1 & 0\\
0 & 0 & 1
\end{pmatrix}\begin{pmatrix}
1 & bx & y\\
0 & 1 & qx\\
0 & 0 & 1
\end{pmatrix}\begin{pmatrix}
1 & \frac{aj}{q} & 0\\
0 & 1 & 0\\
0 & 0 & 1
\end{pmatrix}
= \begin{pmatrix}
1 & bx & y-ajx\\
0 & 1 & qx\\
0 & 0 & 1\end{pmatrix}
\end{equation}
and that, by doubling \(a,b,\) and \(q\) in \eqref{theta1} if necessary, we may assume without loss of generality that \(b\) is even, so that \(x,y-bqx^2/2\in \mathbb{Z}\) if and only if \(x,y\in \mathbb{Z}.\)

Observe that 
\begin{align*}\varepsilon_2(n_{j,k})e^{t_1Y_1}e^{t_2Y_2}\gamma_2(j)&=\begin{pmatrix}
1 & tn_{j,k} & 0\\
0 & 1 & 0\\
0 & 0 & 1\end{pmatrix}
\begin{pmatrix}
1 & bt_1 & \frac{bqt_1^2}{2}\\
0 & 1 & qt_1\\
0 & 0 & 1\end{pmatrix}
\begin{pmatrix}
1 & 0 & t_2\\
0 & 1 & 0\\
0 & 0 & 1\end{pmatrix}
\begin{pmatrix}
1 & \frac{aj}{q} & 0\\
0 & 1 & 0\\
0 & 0 & 1\end{pmatrix}
\\&= \begin{pmatrix}
1 & \frac{aj}{q}+bt_1+tn_{j,k} & \frac{bqt_1^2}{2}+qt_1tn_{j,k}+t_2\\
0 & 1 & qt_1\\
0 & 0 & 1\end{pmatrix},
\end{align*}
so that in Mal'cev coordinates \begin{equation}\label{f2formula}\begin{split}F_2(e^{t_1Y_1}e^{t_2Y_2}\Gamma_2)&=F_1(\varepsilon_2(n_{j,k})e^{t_1Y_1}e^{t_2Y_2}\gamma_2(j)\Gamma_1)\\&= e\left(\frac{bqt_1^2}{2}+qt_1tn_{j,k}+t_2-\left(\frac{aj}{q}+bt_1+tn_{j,k}\right) \lfloor qt_1\rfloor\right).\end{split}\end{equation}
Therefore \begin{align*}
&\int_{G_2/\Gamma_2}\Psi_2F_2 \\&=\int_0^1\int_0^1  e\left(\frac{bqt_1^2}{2}+qt_1tn_{j,k}+t_2-\left(\frac{aj}{q}+bt_1+tn_{j,k}\right) \lfloor qt_1\rfloor\right)\psi(qt_1) ~\mathrm{d}t_1\, \mathrm{d}t_2=0  
\end{align*} since the integral over \(t_2\) is 0.

Since 
\begin{align*}\varepsilon_2(n_{j,k})g_2(n)\gamma_2(j)&=\begin{pmatrix}
1 & tn_{j,k} & 0\\
0 & 1 & 0\\
0 & 0 & 1
\end{pmatrix}\begin{pmatrix}
1 & \frac{b\beta n}{q} & -\beta t n^2\\
0 & 1 &  \beta n\\
0 & 0 & 1
\end{pmatrix}\begin{pmatrix}
1 & \frac{aj}{q} & 0\\
0 & 1 & 0\\
0 & 0 & 1
\end{pmatrix}
\\&=\begin{pmatrix}
1 & \frac{aj+b\beta n}{q}+tn_{j,k}  & \beta nn_{j,k} t-\beta n^2t\\
0 & 1 & \beta n\\
0 & 0 & 1
\end{pmatrix},
\end{align*}
it follows that
\begin{multline*} F_1(\varepsilon_2(n) g_2(n)\gamma_2(j)\Gamma_1) -F_1(\varepsilon_2(n_{j,k})g_2(n)\gamma_2(j)\Gamma_1)\\ =e\left(-\left(\frac{aj+b\beta n}{q}+tn\right)\lfloor \beta n\rfloor\right )-e\left(\beta nn_{j,k} t-\beta n^2t-\left(\frac{aj+b\beta n}{q}+tn_{j,k} \right)\lfloor\beta n\rfloor\right),\end{multline*}
which we may bound by 
\begin{align*}
&\ll |e(-tn\lfloor \beta n\rfloor+tn_{j,k}\lfloor \beta n\rfloor -\beta nn_{j,k}t+\beta n^2t)-1|
\\ &\ll |-tn\lfloor \beta n\rfloor+tn_{j,k}\lfloor \beta n\rfloor -\beta nn_{j,k}t+\beta n^2t|
\\& =|t||\lfloor \beta n\rfloor(n_{j,k}-n)-\beta n (n_{j,k}-n)|
\\ &=|t||\lfloor \beta n\rfloor-\beta n||n_{j,k}-n|.
\end{align*}
Therefore \[E \ll |t||P_{j,k}|q\ll|t|q\delta_2 X\ll \delta_1^{-O(1)}\delta_2.\]
This completes the proof. 
\end{proof}

To derive the same conclusion as above, Neale begins by showing that equidistribution of \((g_2(n)\Gamma_2)_{n\in[X]}\) implies  equidistribution of \((\gamma_2^{-1}g_2(n)\gamma_2(\gamma_2^{-1}\Gamma_2\gamma_2))_{n\in[X]}\) on a `conjugate nilmanifold'  of the shape \((\gamma_2^{-1}G_2\gamma_2)/(\gamma_2^{-1}\Gamma_2\gamma_2).\) Working with such sequences requires extra computation as well as technical machinery of Green and Tao \cite[Appendix A]{GT} to handle the associated Lipschitz norm (see e.g.~\cite[Appendix B]{gtmobius} for a similar manoeuvre). We avoid this entirely by our choice of function \(F_2:G_2/\Gamma_2\rightarrow\mathbb{C}\) which we observe is well-defined  provided that  \(\gamma_2^{-1}\Gamma_2\gamma_2\subseteq \Gamma_1.\) This simplification is similarly applied to \(G_3/\Gamma_3\) given later.

\begin{lemma}[Lack of total equidistribution in \(G_2/\Gamma_2\) implies being in major arcs]\label{delta2equid} Let \(\delta_1,\delta_2\) satisfy \eqref{deltas} and let \(X\) be sufficiently large in terms of \(\beta.\) Suppose \(\theta\) is of the form \eqref{theta1}. If \((g_2(n)\Gamma_2)_{n\in [X]}\) is not totally \(\delta_2\)-equidistributed in \(G_2/\Gamma_2,\) then \begin{equation}\label{theta2}\theta=\frac{a }{q}+t\end{equation} for integers \(a\) and \(q\) with \(1\leq q\ll \delta_1^{-O(1)}\) and \(t\ll \delta_2^{-O(1)}X^{-2}.\)  
\end{lemma}
\begin{proof}
As \((g_2(n)\Gamma_2)_{n\in [X]}\) is not totally \(\delta_2\)-equidistributed in \(G_2/\Gamma_2,\)  by Theorem \ref{GTequid} there exists a non-trivial horizontal character \(\eta:G_2 \rightarrow \mathbb{R}/\mathbb{Z},\) with \(|\eta|\ll \delta_2^{-O(1)},\) such that \(\Vert\eta \circ g_2\Vert_{C^\infty[X]}\ll \delta_2^{-O(1)}.\) Such characters are of the form \(\eta\begin{psmallmatrix}
1 & bx & y\\
0 & 1 & qx\\
0 & 0 & 1
\end{psmallmatrix}=k_3x+k_4(y-bqx^2/2)\) for integers \(k_3,k_4\ll\delta_2^{-O(1)} \) not both zero. Since \begin{align*}
\eta(g_2(n))&=k_3\frac{\beta  n}{q}-k_4\left(\beta  tn^2+\frac{bq}{2}\left(\frac{\beta  n}{q}\right)^2\right)\\  &= \left(k_3\frac{\beta  }{q}-k_4\beta t -k_4\frac{b\beta^2 }{2q} \right)n-2k_4\left(\beta  t+\frac{b\beta ^2 }{2q}\right)\binom{n}{2},
\end{align*}
we see that \begin{equation}\label{eqn1}\left\Vert k_3\frac{\beta  }{q}-k_4\beta  t -k_4\frac{b\beta^2 }{2q}\right\Vert_{\mathbb{R}/\mathbb{Z}}\ll \frac{\delta_2^{-O(1)}}{X}\end{equation} and 
\begin{equation}\label{eqn2} \left\Vert -2k_4\left(\beta  t+\frac{b\beta^2 }{2q}\right)\right\Vert_{\mathbb{R}/\mathbb{Z}}\ll \frac{\delta_2^{-O(1)}}{X^2}.\end{equation}

We first claim that this implies  \(k_4\neq 0.\) Indeed,
 if \(k_4=0,\) then \(k_3\neq 0\) and so \eqref{eqn1} implies
\[\frac{1}{q}\left\Vert k_3\beta  \right\Vert_{\mathbb{R}/\mathbb{Z}}\leq \left\Vert k_3\frac{\beta }{q}\right\Vert_{\mathbb{R}/\mathbb{Z}}\ll \frac{\delta_2^{-O(1)}}{X}\] which contradicts   \(\beta\) being irrational and non-Liouville as, for some positive integer \(M=M(\beta)\leq 2\mu\), we have 
\[\frac{1}{q}\left\Vert k_3\beta  \right\Vert_{\mathbb{R}/\mathbb{Z}}\gg_\beta\frac{1}{q|k_3|^{M-1}}\gg \delta_1^{O(1)}\delta_2^{O(M)}.\]  Hence \(k_4\neq 0\) as claimed.

Next, we claim that \(b=0\) because  \(\beta^2\) is irrational and non-Liouville. Indeed, by the triangle inequality 
\[\left\Vert -2k_4\left(\beta  t+\frac{b\beta^2 }{2q}\right)\right\Vert_{\mathbb{R}/\mathbb{Z}}\geq\left\Vert k_4\frac{b\beta^2  }{q}\right\Vert_{\mathbb{R}/\mathbb{Z}}-\left\Vert 2k_4\beta  t\right\Vert_{\mathbb{R}/\mathbb{Z}}.\] Since \(t\ll\delta_1^{-O(1)}X^{-1}\) and \(k_4\ll \delta_2^{-O(1)},\)  we have  \begin{equation}\label{smallt}\left\Vert 2k_4\beta  t\right\Vert_{\mathbb{R}/\mathbb{Z}}=| 2k_4\beta  t|\ll \frac{\delta_1^{-O(1)}\delta_2^{-O(1)}}{X}.\end{equation}
Then \eqref{eqn2} gives \[\frac{1}{q}\left\Vert k_4b\beta^2 \right\Vert_{\mathbb{R}/\mathbb{Z}}\leq\left\Vert k_4\frac{b\beta^2  }{q}\right\Vert_{\mathbb{R}/\mathbb{Z}}\ll \frac{\delta_2^{-O(1)}}{X^2}+\frac{\delta_1^{-O(1)}\delta_2^{-O(1)}}{X}.\] However \(\beta^2\) is irrational and non-Liouville, so if \(k_4b\beta^2\neq 0\) then there exists some positive integer \(M'=M'(\beta^2)\leq 2\mu\) such that 
\[\frac{1}{q}\left\Vert k_4b \beta^2 \right\Vert_{\mathbb{R}/\mathbb{Z}}\gg_\beta \frac{1}{q|k_4b|^{M'-1}}\gg \delta_1^{O(M')}\delta_2^{O(M')},\] which is incompatible with the previous bound.  Therefore \(b=0\) as claimed.
We may now strengthen \eqref{smallt}, since \eqref{eqn2} implies
\[| 2k_4\beta  t|=\left\Vert 2k_4\beta  t\right\Vert_{\mathbb{R}/\mathbb{Z}}\ll\frac{\delta_2^{-O(1)}}{X^2},\] and therefore \[t\ll\frac{\delta_2^{-O(1)}}{X^2},\] completing the proof.
\end{proof}
\section{Equidistribution in \texorpdfstring{$G_3/\Gamma_3$}{G3/Γ3}}\label{sectiong3}
Let \(G_3=\left\{\begin{psmallmatrix}
1 & 0 & 0\\
0 & 1 & qx\\
0 & 0 & 1
\end{psmallmatrix}: x\in \mathbb{R}\right\}\) and  \(\Gamma_3=\left\{\begin{psmallmatrix}
1 & 0 & 0\\
0 & 1 & qx\\
0 & 0 & 1
\end{psmallmatrix}: x\in\mathbb{Z}\right\},\)
with Mal'cev basis given by
\(Z_1=\begin{psmallmatrix}
0 & 0 & 0\\
0 & 0 & q\\
0 & 0 & 0
\end{psmallmatrix}.\) This basis is clearly \(1\)-rational, and in fact all the structure constants equal \(0.\)  Let \(\theta\) be as in \eqref{theta2}. We can once again factorise the polynomial  
\begin{align*}g_1(n)&=\begin{pmatrix}
1 & \theta n & 0\\
0 & 1 & \beta n\\
0 & 0 & 1
\end{pmatrix}
\\&=\begin{pmatrix}
1 & \left(\frac{a}{q}+t\right) n & 0\\
0 & 1 & \beta n\\
0 & 0 & 1
\end{pmatrix}
\\&=\begin{pmatrix}
1 & tn & -\beta n^2t\\
0 & 1 & 0\\
0 & 0 & 1
\end{pmatrix}\begin{pmatrix}
1 & 0 & 0\\
0 & 1 & \beta n\\
0 & 0 & 1
\end{pmatrix}\begin{pmatrix}
1 & \frac{an}{q} & 0\\
0 & 1 & 0\\
0 & 0 & 1
\end{pmatrix}
\\&=\varepsilon_3(n)g_3(n)\gamma_3(n), 
\end{align*}
where \(\varepsilon_3(n)\) is `smooth', \(g_3(n)\in G_3\) is a polynomial, and \(\gamma_3(n)\Gamma_1\) is `periodic'. 
\begin{lemma}[Total equidistribution in \(G_3/\Gamma_3\)] \label{equidg3} Let \(\delta_1,\delta_2,\delta_3\) satisfy \eqref{deltas} and let \(X\) be sufficiently large in terms of \(\beta.\) Then \((g_3(n)\Gamma_3)_{n\in [X]}\) is totally \(\delta_3\)-equidistributed in \(G_3/\Gamma_3.\) 
\end{lemma}
\begin{proof}
Suppose \((g_3(n)\Gamma_3)_{n\in [X]}\) were not totally \(\delta_3\)-equidistributed in \(G_3/\Gamma_3.\) Then by Theorem \ref{GTequid}, there exists a non-trivial horizontal character \(\eta:G_3\rightarrow \mathbb{R}/\mathbb{Z}\) with \(|\eta|\ll \delta_3^{-O(1)},\)   such that \(\Vert\eta \circ g_3\Vert_{C^\infty[X]}\ll \delta_3^{-O(1)}.\) Such characters are of the form \(\eta\begin{psmallmatrix}
1 & 0 & 0\\
0 & 1 & qx\\
0 & 0 & 1
\end{psmallmatrix}=k_5x\) for a non-zero integer \(k_5\ll\delta_3^{-O(1)}.\)  
Since \(\eta(g_3(n))=k_5\beta n/q ,\) we see that 
\[\frac{1}{q}\left\Vert k_5 \beta \right\Vert_{\mathbb{R}/\mathbb{Z}}\leq \left\Vert \frac{k_5 \beta}{q} \right\Vert_{\mathbb{R}/\mathbb{Z}}\ll \frac{\delta_3^{-O(1)}}{X}.\] 

However, since \(\beta\) is irrational and non-Liouville, for some positive integer \(M=M(\beta)\leq 2\mu\) we have
\[\frac{1}{q}\Vert k_5 \beta  \Vert_{\mathbb{R}/\mathbb{Z}}\gg_\beta \frac{1}{q|k_5|^{M-1}} \gg \delta_1^{O(1)}\delta_3^{O(M)}.\] This is impossible for large \(X\)  by our choice of \(\delta_1\) and \(\delta_3\) in \eqref{deltas}.
\end{proof}
Next we obtain a major arc estimate, whose main term  should be  compared to the main term \eqref{root2main} when \(\beta=\sqrt{2},\) or Lemma \ref{majorarcrootestimate} below. 
\begin{lemma}[Major and minor arc estimates]\label{majorarcest} Suppose \(\beta\) and \(\beta^2\) are  real, irrational, and non-Liouville. Let \(\delta_1,\delta_2,\delta_3\) satisfy \eqref{deltas} and let \(X\) be sufficiently large in terms of \(\beta.\) Assume \(\theta\in \mathbb{R}.\) Then 
\[\mathbb{E}_{n\in [X]}e(-\theta n \lfloor \beta  n\rfloor)\ll \delta_1^{1/2},\] or  \(\theta\) is of the form \eqref{theta2}. Furthermore, for all \(\theta\) as in \eqref{theta2} with \((a,q)=1\) we have 
\begin{align*}
\mathbb{E}_{n\in [X]}e(-\theta n \lfloor \beta  n\rfloor)=\frac{1}{qX}\int_{0}^{X}e(-\beta v^2 t) ~\mathrm{d}v+O(\delta_2^{-O(1)}\delta_3^{1/2}).
\end{align*}
\end{lemma}
\begin{proof}
By the work in the Sections \ref{sectiong1} and \ref{sectiong2}, we may assume that \(\theta\) is of the form \eqref{theta2}. The coprimality of \(a\) and \(q\) is not used until the very end of the proof. By Lemma \ref{equidg3}, \((g_3(n)\Gamma_3)_{n\in [X]}\) is totally \(\delta_3\)-equidistributed in \(G_3/\Gamma_3.\) Therefore, 
 \[\left|\mathbb{E}_{n\in P} \Psi_3F_3(g_3(n)\Gamma_3)-\int_{G_3/\Gamma_3} \Psi_3F_3\right|\leq \delta_3\Vert \Psi_3 F_3 \Vert_{\textrm{Lip}}\] for all arithmetic progressions \(P\subseteq [X]\) of size \(|P|\geq \delta_3X,\) where \(\Psi_3,F_3:G_3/\Gamma_3\rightarrow\mathbb{C}\) are given by \(F_3(g_3\Gamma_3)=F_1(\varepsilon_3g_3\gamma_3\Gamma_1),\) for some specially chosen \(\varepsilon_3,\gamma_3\in G_3,\) and \(\Psi_3(g_3\Gamma_3)=\Psi_1(\varepsilon_3g_3\gamma_3\Gamma_1)\) with \(\varepsilon=\delta_3^{1/2}.\) As before, the functions \(\Psi_3\) and \(F_3\) are  well-defined  provided \(\gamma_3^{-1}\Gamma_3\gamma_3\subseteq \Gamma_1.\) 
 By writing in terms of Mal'cev coordinates (see below), we have \(\int_{G_3/\Gamma_3}(1-\Psi_3)=\int_0^1 1-\psi(qt_1)\mathrm{d}t_1\ll \delta_3^{1/2},\) and we also find that  \(\Vert \Psi_3 \Vert_{\textup{Lip}},\Vert \Psi_3 F_3\Vert_{\textup{Lip}}\ll \delta_1^{-O(1)}\delta_3^{-1/2}\)  by the discussion at the end of Section \ref{nilbackground}.
 
Partition the interval \([X]\) into arithmetic progressions  \(P_{j,k}\) of common difference \(q\) with elements congruent to \(j\) modulo \(q,\)  with each progression of size \( |P_{j,k}|=\lceil\delta_3 X \rceil,\) and some leftover set of size \(O(\delta_3 X q).\)
Pick an element \(n_{j,k}\in P_{j,k}.\)   By a similar analysis to the previous section,
\begin{align*}
\mathbb{E}_{n\in [X]} F_1(g_1(n)\Gamma_1)&=\mathbb{E}_{n\in [X]}F_1(\varepsilon_3(n)g_3(n)\gamma_3(n)\Gamma_1)
\\&=\mathbb{E}_{j,k}\mathbb{E}_{n\in P_j,k}F_1(\varepsilon_3(n)g_3(n)\gamma_3(j)\Gamma_1)+O(\delta_3 Xq/X)
\\&=\mathbb{E}_{j,k}\mathbb{E}_{n\in P_j,k}F_1(\varepsilon_3(n_{j,k})g_3(n)\gamma_3(j)\Gamma_1)+O(\delta_1^{-O(1)}\delta_3)+E'
\\&=\mathbb{E}_{j,k}\mathbb{E}_{n\in P_j,k}F_3(g_3(n)\Gamma_3)+O(\delta_1^{-O(1)}\delta_3)+E'
\\&=\mathbb{E}_{j,k}\mathbb{E}_{n\in P_j,k}\Psi_3F_3(g_3(n)\Gamma_3)+O(\delta_1^{-O(1)}\delta_3^{1/2})+E'
\\&=\mathbb{E}_{j,k}\int_{G_3/\Gamma_3}\Psi_3F_3+O(\delta_1^{-O(1)}\delta_3^{1/2})+E'
\\&=\mathbb{E}_{j,k}\int_{G_3/\Gamma_3}F_3+O(\delta_1^{-O(1)}\delta_3^{1/2})+E'.
\end{align*}
Here
\[E'=\max_{j,k}\left|\mathbb{E}_{n\in P_{j,k}}F_1(\varepsilon_3(n)g_3(n)\gamma_3(j)\Gamma_1)-F_1(\varepsilon_3(n_{j,k})g_3(n)\gamma_3(j)\Gamma_1)\right|\] is the error incurred by replacing \(n\) with a representative element \(n_{j,k}\) in the `smooth' part, and where we choose \(\varepsilon_3=\varepsilon_3(n_{j,k})\) and \(\gamma_3=\gamma_3(j)\) for the definition of \(\Psi_3\) and \(F_3\) (and although \(\Psi_3\) and \(F_3\) depend on \(j\) and \(k\) for example, we suppress this in our notation). One can see  that this \(\gamma_3\) satisfies the property \(\gamma_3^{-1}\Gamma_3\gamma_3\subseteq \Gamma_1\) as before in \eqref{conjugation}, and so our \(\Psi_3\) and \(F_3\) are well-defined. 

Since 
\begin{align*}\varepsilon_3(n_{j,k})g_3(n)\gamma_3(j)&=\begin{pmatrix}
1 & tn_{j,k} & -\beta n_{j,k}^2t\\
0 & 1 & 0\\
0 & 0 & 1
\end{pmatrix}\begin{pmatrix}
1 & 0 & 0\\
0 & 1 & \beta n\\
0 & 0 & 1
\end{pmatrix}\begin{pmatrix}
1 & \frac{aj}{q} & 0\\
0 & 1 & 0\\
0 & 0 & 1
\end{pmatrix}
\\&=\begin{pmatrix}
1 & \frac{aj}{q}+tn_{j,k} & \beta n  n_{j,k} t-\beta n_{j,k}^2t\\
0 & 1 & \beta n\\
0 & 0 & 1
\end{pmatrix},
\end{align*}
it follows that 
\begin{multline*} F_1(\varepsilon_3(n) g_3(n)\gamma_3(j)\Gamma_1) -F_1(\varepsilon_3(n_{j,k})g_3(n)\gamma_3(j)\Gamma_1)\\ =e\left( -\left(\frac{aj}{q}+tn\right)\lfloor \beta n\rfloor\right)-e\left(\beta n  n_{j,k} t-\beta n_{j,k}^2t-\left(\frac{aj}{q}+tn_{j,k}\right)\lfloor \beta n\rfloor\right),\end{multline*}
which we may bound by 
\begin{align*}
&\ll |e(-tn\lfloor \beta n\rfloor+tn_{j,k}\lfloor \beta n\rfloor-\beta n  n_{j,k} t+\beta n_{j,k}^2t)-1|
\\ &\ll |-tn\lfloor \beta n\rfloor+tn_{j,k}\lfloor \beta n\rfloor-\beta n  n_{j,k} t+\beta n_{j,k}^2t|
\\ &=|t||(n_{j,k}-n)\lfloor \beta n\rfloor+\beta n_{j,k}(n_{j,k}-n)|
\\&=|t||n_{j,k}-n||\lfloor \beta n\rfloor+\beta n|.
\end{align*}
Therefore \[E' \ll |\beta||t||P_{j,k}|qX\ll |\beta||t|q\delta_3 X^2 \ll \delta_2^{-O(1)}\delta_3 .\]

Now we consider the integral. Observe that 
\begin{align*}\varepsilon_3(n_{j,k})e^{t_1Z_1}\gamma_3(j)&=\begin{pmatrix}
1 & tn_{j,k} & -\beta n_{j,k}^2t\\
0 & 1 & 0\\
0 & 0 & 1\end{pmatrix}
\begin{pmatrix}
1 & 0 & 0\\
0 & 1 & qt_1\\
0 & 0 & 1\end{pmatrix}
\begin{pmatrix}
1 & \frac{aj}{q} & 0\\
0 & 1 & 0\\
0 & 0 & 1\end{pmatrix}
\\&= \begin{pmatrix}
1 & \frac{aj}{q}+tn_{j,k} & qt_1tn_{j,k}-\beta n_{j,k}^2t\\
0 & 1 & qt_1\\
0 & 0 & 1\end{pmatrix},
\end{align*}
so that in Mal'cev coordinates \begin{equation}\label{f3formula}
\begin{split}F_3(e^{t_1Z_1}\Gamma_3)&=F_1(\varepsilon_3(n_{j,k})e^{t_1Z_1}\gamma_3(j)\Gamma_1)\\&= e\left(qt_1tn_{j,k}-\beta n_{j,k}^2t-\left(\frac{aj}{q}+tn_{j,k}\right)\lfloor q t_1\rfloor\right)\\&= e\left(-\beta n_{j,k}^2t-\frac{aj}{q}\lfloor q t_1\rfloor\right)+O(|t|qX),\end{split}\end{equation} where the last line follows from Taylor's theorem (assuming \(0\leq t_1<1\)). 
Therefore \[\int_{G_3/\Gamma_3}F_3=\int_0^1  e\left(-\beta n_{j,k}^2t-\frac{aj}{q}\lfloor q t_1\rfloor\right)~\mathrm{d}t_1+O(|t|qX). \]
Splitting the interval into regions where \(\lfloor qt_1 \rfloor\) equals a constant \(m,\) we see that the right-hand integral  is
\[\sum_{0\leq m<q}\int_{m/q}^{(m+1)/q}  e\left(-\beta n_{j,k}^2t-\frac{ajm}{q}\right) ~\mathrm{d}t_1=e(-\beta n_{j,k}^2t)\mathbb{E}_{m}e \left(\frac{-ajm}{q}\right).
\]
Taking the average over \(j,k\) yields 
\[\mathbb{E}_{j,k}\int_{G_3/\Gamma_3}F_3=\mathbb{E}_{j,k}e(-\beta n_{j,k}^2t)\mathbb{E}_{m}e \left(\frac{-ajm}{q}\right)+O(|t|qX).\]

We may choose \(P_{j,k}\) such that we can take as our representative element \(n_{j,k}=kq\lceil \delta_3 X \rceil +O(q)\) where \(0\leq j< q\) and \(1\leq k\leq X/(q\lceil \delta_3 X \rceil ).\) 
Then we can approximate \(n_{j,k}^2=(kq\lceil \delta_3 X \rceil)^2 +O(qn_{j,k})=(kq\lceil \delta_3 X \rceil)^2 +O(qX),\) and so
by Taylor expansion \begin{align*}\mathbb{E}_{k}e(-\beta n_{j,k}^2t)&=\mathbb{E}_{k}e(-\beta (kq\lceil \delta_3 X \rceil)^2t)+O(|\beta||t|qX)
\\& =\frac{q\lceil \delta_3 X \rceil}{X}\sum_{ k\leq X/(q\lceil \delta_3 X \rceil )} e(-\beta (kq\lceil \delta_3 X \rceil )^2 t)+O(|\beta||t|qX)
\\&= \frac{q\lceil \delta_3 X \rceil}{X}\int_{0}^{X/(q\lceil \delta_3 X \rceil )}e(-\beta (qu\lceil \delta_3 X \rceil )^2 t) ~\mathrm{d}u +O(|\beta||t|q\delta_3 X^2)
\\& =\frac{1}{X}\int_{0}^{X}e(-\beta v^2 t) ~\mathrm{d}v +O(\delta_2^{-O(1)}\delta_3).
\end{align*}
Here we have approximated the sum by an integral using
\begin{multline*}\sum_{k\leq K}f(k)=\sum_{k\leq K}\int_{k-1}^kf(k)~\mathrm{d}u=\sum_{k\leq K}\int_{k-1}^kf(u)+O\left(\max_{v\in[k-1,k]}|f'(v)|\right)~\mathrm{d}u\\=\int_0^Kf(u) ~\mathrm{d}u+O\left(K\max_{v\in[0,K]}|f'(v)|\right).\end{multline*} 

Assembling these estimates gives
\[\mathbb{E}_{n\in [X]}e(-\theta n \lfloor \beta  n\rfloor)=\mathbb{E}_{j,m ~(\textrm{mod } q)}e \left(\frac{-ajm}{q}\right)\frac{1}{X}\int_{0}^{X}e(-\beta v^2 t) ~\mathrm{d}v+O(\delta_2^{-O(1)}\delta_3^{1/2}).\]
We have not yet used the assumption that \((a,q)=1.\) With this coprimality condition, the average over residues modulo \(q\) can be evaluated by separately considering  \(j=0\) and \(j\neq 0\) (and computing a geometric sum) for the desired result. \end{proof}

Though we do not need the next result, we record it for the interested reader. Suppose instead that \(\beta\in \mathbb{R}\setminus\mathbb{Q}\) and  \(\beta^2\in \mathbb{Q}.\) As shown by Neale \cite[\S6]{neale}, the major arcs are then given by   \begin{equation}\label{majorsqroot}\theta=\frac{a+b\beta }{q}+t\end{equation} for integers \(a,b,q\) with \(b,q\ll \delta_1^{-O(1)}, q\geq 1,\) and \(t\ll \delta_2^{-O(1)}X^{-2}.\) The accompanying estimate below was proved by Neale  for the case \(\beta=\sqrt{2}\)  (see work of Daskalakis \cite{dask} for the general case).
\begin{lemma}[Major and minor arc estimates for square roots of rationals]\label{majorarcrootestimate} Suppose  \(\beta \in \mathbb{R}\setminus  \mathbb{Q}\)  satisfies \(\beta^2=c/d\) for some  positive \(c,d\in \mathbb{Z}.\) 
 Let \(\delta_1,\delta_2,\delta_3\) satisfy \eqref{deltas} and let \(X\) be sufficiently large in terms of \(\beta.\) Assume \(\theta\in \mathbb{R}.\) Then 
\[\mathbb{E}_{n\in [X]}e(-\theta n \lfloor \beta  n\rfloor)\ll \delta_1^{1/2},\] or  \(\theta\) is of the form \eqref{majorsqroot}. Furthermore, for all \(\theta\) as in \eqref{majorsqroot} we have
\begin{multline*}
\mathbb{E}_{n\in [X]}e(-\theta n \lfloor \beta  n\rfloor)=\\ \mathbb{E}_{\substack{j~(\textup{mod } 2dq)\\ m~(\textup{mod } 2q)}}e\left(\frac{-ajm}{q}-\frac{b}{2q}(\beta^2j^2+m^2)\right)\int_0^1 e\left(\frac{b}{2q}x^2\right)~\mathrm{d}x\frac{1}{X}\int_{0}^{X}e(-\beta v^2 t) ~\mathrm{d}v\\+O(\delta_2^{-O(1)}\delta_3^{1/2}).
\end{multline*}
\end{lemma}
\begin{proof} See \cite{neale} or \cite{dask}.  For the convenience of the reader, we  explain  the modifications required to the method we have presented so far. For the analogue of Lemma \ref{delta2equid}, we instead use that \(\beta^2\in \mathbb{Q}\) and \(\Vert \alpha\Vert_{\mathbb{R}/\mathbb{Z}}/q\leq\Vert \alpha/q\Vert_{\mathbb{R}/\mathbb{Z}}\) for integers \(q\geq 1\) to effectively remove the contribution from the \(\beta^2\) term in \eqref{eqn1} and \eqref{eqn2}.    Crucially, we no longer deduce that \(b=0,\) and hence the major arcs are of a different shape.   

Working with  \(G_3=\left\{\begin{psmallmatrix}
1 & bx & \frac{bqx^2}{2}\\
0 & 1 & qx\\
0 & 0 & 1
\end{psmallmatrix}: x\in \mathbb{R}\right\}\) and  \(\Gamma_3=\left\{\begin{psmallmatrix}
1 & bx & \frac{bqx^2}{2}\\
0 & 1 & qx\\
0 & 0 & 1
\end{psmallmatrix}: x\in\mathbb{Z}\right\},\)
with Mal'cev basis given by
\(Z_1=\begin{psmallmatrix}
0 & b & 0\\
0 & 0 & q\\
0 & 0 & 0
\end{psmallmatrix},\)
we write
\begin{align*}g_1(n)=\begin{pmatrix}
1 & tn & -\beta n^2t\\
0 & 1 & 0\\
0 & 0 & 1
\end{pmatrix}\begin{pmatrix}
1 & \frac{b\beta n}{q} & \frac{b\beta^2n^2}{2q}\\
0 & 1 & \beta n\\
0 & 0 & 1
\end{pmatrix}\begin{pmatrix}
1 & \frac{an}{q} & -\frac{b\beta^2n^2}{2q}\\
0 & 1 & 0\\
0 & 0 & 1
\end{pmatrix}=\varepsilon_3(n)g_3(n)\gamma_3(n). 
\end{align*}
 We may choose \(d\) minimal (so that \(X\) is large in terms of  \(d\)) since  any positive integer multiple of this choice gives the same main term as in the lemma. We take the range \(0 \leq j<2dq\) because \(\gamma_3(n)\Gamma_1\) is `periodic' modulo \(2dq.\) To ensure that \(\gamma_3^{-1}\Gamma_3\gamma_3\subseteq\Gamma_1,\) we insist that \(bq\) is even (which was automatically true before because \(b=0\)) by working with the double of \(a,b,\) and \(q\) (just as we did in the proof of Lemma \ref{totalequidg2}). This factor of 2 cancels in the final main term, except for the range of \(m\) in the expectation which is now \(0\leq m<2q.\) The rest of the computation of \(\int_{G_3/\Gamma_3}F_3\) is similar to before, although one finds that there is an extra integral to evaluate by completing the square. This introduces an additional term in the expectation and the Fresnel integral given in the lemma. 
\end{proof}

\section{Positivity at major arcs}\label{sectionpositive}
We are now ready to prove Lemma \ref{nonnegative}, thereby proving Theorem \ref{mainthm1}. Let \(\delta_1,\delta_2,\delta_3\) satisfy \eqref{deltas}.
We have so far shown that for some small absolute constant \(c'>0,\) and for \(X\)  large in terms of \(\beta,\)  we have cancellation \[\left|\mathbb{E}_{n\in [X]}e(-\theta n \lfloor \beta  n\rfloor)\right|\leq X^{-c'/\mu},\]}in which case we are done, or  we have the major arc estimate in Lemma \ref{majorarcest}.
Consider the real part of the integral appearing in the major arc estimate
\[\Re\int_{0}^{X}e(-\beta v^2 t) ~\mathrm{d}v=\int_{0}^{X}\cos(2\pi|\beta t|v^2 ) ~\mathrm{d}v,\] with \(t\ll \delta_2^{-O(1)}X^{-2}.\) If \(|\beta t|X^2<1/100,\) then the integrand is \(\gg 1,\) and thus the integral is \(\gg X.\) If instead \(|\beta  t|X^2\geq 1/100\) then we substitute \(w=\sqrt{|\beta t|}v\) to obtain 
\[\frac{1}{\sqrt{|\beta t|}}\int_0^{\sqrt{|\beta t|}X}\cos(2 \pi w^2)  ~\mathrm{d}w\gg \frac{1}{\sqrt{|\beta t|}}\gg \frac{\delta_2^{O(1)}X}{\sqrt{|\beta|}}. \] Here we have used the fact that the integral is a Fresnel integral, and is known classically to be  \(\gg 1\) since the upper limit is at least \(1/100.\) 
Indeed,  
\[\int_0^{\sqrt{1/4}} \cos(2\pi w^2)~\mathrm{d}w\geq 0.38 \quad \textrm{and} \quad  \left| \int_{\sqrt{1/4}}^{\sqrt{3/4}} \cos(2\pi w^2)~\mathrm{d}w\right| \leq 0.23,\] and for all positive integers \(k,\) one can verify (e.g.~by substituting \(x=w^2\)) that the integral between the zeros \(\sqrt{(4k-1)/4}\) and \(\sqrt{(4k+1)/4}\) is at least as large as the absolute value of the integral between the zeros \(\sqrt{(4k+1)/4}\) and \(\sqrt{(4k+3)/4}.\) 

Therefore, by Lemma \ref{majorarcest},  we in fact deduce a major arc bound of
\[\Re \mathbb{E}_{n\in [X]}e(-\theta n \lfloor \beta  n\rfloor)\geq X^{-c'/\mu}\] (possibly upon adjusting \(c'\)) for  \(\theta\) of the form \eqref{theta2},
 which completes the proof of (slightly more precise versions of) Lemma \ref{nonnegative} and Theorem \ref{mainthm1}. 

\section{Bracket polynomials}\label{bracketpolysection}
In this section, we generalise the previous argument, replacing \(n\lfloor \beta n\rfloor\) with 
\[P(n)-Q(n)\lfloor R(n) \rfloor,\] where
\(Q,R\) are as in  Theorem \ref{polythm} (we initially handle the case  \(R^2\in \mathbb{L}'[n]\setminus\mathbb{Q}[n]\)), and \(P\in \mathbb{Q}[n]\) such that \(P:\mathbb{Z}\rightarrow \mathbb{Z}\) and \(\deg P\leq \deg Q, \deg R.\)
In fact, we eventually specialise to  \(P=-\ell Q\) for some \(\ell\in \mathbb{Z},\) but we choose to  work with the less restricted scenario until the very end in order to attain more general minor and major arc estimates. 
Let \(X\) be sufficiently large in terms of \(P,Q,\) and \(R.\)    Let \(\kappa_1,\kappa_2,\kappa_3>0\) be  constants which are sufficiently small  in terms of \(\deg Q,\deg R\) (which is needed because of  the parameters in Theorem \ref{GTequid}). Let \(\mu'\geq 1\) be an upper bound for all  the irrationality exponents of the coefficients of both \(R\) and \(R^2\) (which are all finite by assumption).  Once again, we assume that \(\kappa_1\) is sufficiently small in terms of \(\kappa_2,\) and that \(\kappa_2\) is sufficiently small in terms of \(\kappa_3.\) Let
 \begin{equation}\label{deltas2}\delta_1=X^{-\kappa_1/\mu'},\delta_2=X^{-\kappa_2/\mu'}, \textrm{ and } \delta_3=X^{-\kappa_3/\mu'}.\end{equation}
 As before, we take \(\theta\) to be real.
\subsection{Equidistribution in \texorpdfstring{$G_1/\Gamma_1$}{G1/Γ1}, II}
Let  \(\widetilde{g_1}(n)=\begin{psmallmatrix}
1 & \theta Q(n) & \theta P(n)\\
0 & 1 & R(n)\\
0 & 0 & 1
\end{psmallmatrix},\) and let \(F_1:G_1/\Gamma_1\rightarrow\mathbb{C}\) be as before.  The analogue of Lemma \ref{equidG1} is identical. The argument in Lemma \ref{delta1equid} gives
\[\Vert \eta \circ \widetilde{g_1}\Vert_{C^\infty[X]}\ll \delta_1^{-O(1)}\] where 
\[\eta(\widetilde{g_1}(n))=k_1\theta Q(n)+k_2R(n)\]
for integers \(k_1,k_2\ll \delta_1^{-O(1)}\) not both zero.
Write \(Q(n)=\sum_{i\geq 1}q_i\binom{n}{i}\) and \(R(n)=\sum_{i\geq 1}r_i\binom{n}{i}.\) Since \(R\not \in \mathbb{Q}[n],\) there is some \(i\) for which \(r_i\) is irrational. For the largest such \(i,\) by considering the coefficient of \(n^i\) we must have \(r_i\) also non-Liouville  with irrationality exponent at most \(\mu'\)  (because the irrationality exponent of a non-Liouville number does not change when multiplied by a non-zero rational or added to a rational). Because of such a coefficient \(r_i,\) we once again have \(k_1\neq 0.\) 
Therefore, for all \(i\geq 1,\) \begin{equation}\label{thetamulti} q_i\theta=\frac{a_i-k_2r_i}{k_1}+t_i'\end{equation} for some integers \(a_i\)  and  some real \(t_i'\ll\delta_1^{-O(1)}X^{-i}.\)
Writing \(a(n)=\sum_{i\geq 1} a_i \binom{n}{i}\) and \(t(n)=\sum_{i\geq 1} t_i' \binom{n}{i},\) we also see that 
\begin{equation}\label{thetapolyform}\theta Q(n)+\frac{k_2}{k_1}R(n)=\frac{1}{k_1}a(n)+t(n).\end{equation}

\subsection{Equidistribution in \texorpdfstring{$G_2/\Gamma_2$}{G2/Γ2}, II}\label{gamma2ii}
Let \(U=\deg Q,\) so that \(q_U\neq 0\) in \eqref{thetamulti}. Supposing \(\theta\) satisfies \eqref{thetamulti} and \eqref{thetapolyform}, we 
factorise
\begin{align*}\widetilde{g_1}(n)&=\begin{pmatrix}
1 & \theta Q(n) & \theta P(n)\\
0 & 1 & R(n)\\
0 & 0 & 1
\end{pmatrix}
\\&=\begin{pmatrix}
1 & \frac{-k_2R(n)+a(n)}{k_1}+t(n) & \frac{1}{q_U}\left(\frac{a_U-k_2r_U}{k_1}+t_U'\right)P(n)\\
0 & 1 &  R(n)\\
0 & 0 & 1
\end{pmatrix}
\\&=\begin{pmatrix}
1 & t(n) & \frac{t_U'P(n)}{q_U}\\
0 & 1 & 0\\
0 & 0 & 1
\end{pmatrix}\begin{pmatrix}
1 & \frac{-k_2 R(n)}{k_1} & \frac{-k_2r_U P(n)}{k_1q_U}-t(n)R(n)\\
0 & 1 & R(n)\\
0 & 0 & 1
\end{pmatrix}\begin{pmatrix}
1 & \frac{a(n)}{k_1} & \frac{a_U P(n)}{k_1q_U}\\
0 & 1 & 0\\
0 & 0 & 1
\end{pmatrix}\\
\\&=\widetilde{\varepsilon_2}(n)\widetilde{g_2}(n)\widetilde{\gamma_2}(n).
\end{align*}
Define \(G_2,\Gamma_2\) and Mal'cev basis elements \(Y_1,Y_2\) as before, but with the parameters \(b\) and \(q\) replaced with \(-k_2\) and \(k_1\) respectively. 
The analogue of Lemma \ref{totalequidg2} holds with some modifications which we now explain.  Our version of \eqref{conjugation}, showing that \(\widetilde{\gamma_2}(j)^{-1}\Gamma_2\widetilde{\gamma_2}(j)\subseteq \Gamma_1,\) follows from 
\begin{align*}
\begin{pmatrix}
1 & \frac{-a(j)}{k_1} & \frac{-a_U P(j)}{k_1q_U}\\
0 & 1 & 0\\
0 & 0 & 1
\end{pmatrix}\begin{pmatrix}
1 & -k_2x & y\\
0 & 1 & k_1x\\
0 & 0 & 1
\end{pmatrix}&\begin{pmatrix}
1 & \frac{a(j)}{k_1} & \frac{a_U P(j)}{k_1q_U}\\
0 & 1 & 0\\
0 & 0 & 1
\end{pmatrix}
= \begin{pmatrix}
1 & -k_2x & y-a(j)x\\
0 & 1 & k_1x\\
0 & 0 & 1\end{pmatrix}
\end{align*} and the fact that \(a(j)\) is an integer (and, as before, we may assume that \(k_2\) is even). Observe that \(\widetilde{\gamma_2}(n)\Gamma_1\) is `periodic' with period \(q^*\) defined as follows.
Let \(q^*\) be a positive integer such that \(a(n)/k_1=a^*(n)/q^*\) and \(a_U P(n)/(k_1q_U)=P^*(n)/q^*\) with \(a^*,P^*\in \mathbb{Z}[n].\)
If \(n\equiv j \pmod {q^*}\) then  \[\frac{a(n)}{k_1}=\frac{a^*(n)}{q^*}\equiv \frac{a^*(j)}{q^*}=\frac{a(j)}{k_1} \pmod 1,\] and similarly \[\frac{a_U P(n)}{k_1q_U}=\frac{ P^*(n)}{q^*}\equiv \frac{P^*(j)}{q^*}=\frac{a_U P(j)}{k_1q_U} \pmod 1.\] It follows that if \(n\equiv j \pmod {q^*}\) then \(F_1(\widetilde{\varepsilon_2}(n)\widetilde{g_2}(n)\widetilde{\gamma_2}(n)\Gamma_1)=F_1(\widetilde{\varepsilon_2}(n)\widetilde{g_2}(n)\widetilde{\gamma_2}(j)\Gamma_1).\) We again consider short  progressions \(P_{j,k}\) of length \(\asymp \delta_2 X,\) but this time with common difference \(q^*,\) with elements congruent to \(j\) modulo \(q^*.\) It is not too difficult to verify  that we may take \(q^*\ll \delta_1^{-O(1)}.\)

Furthermore, if we denote
\[\widetilde{E}=\max_{j,k}\left|\mathbb{E}_{n\in P_{j,k}}F_1(\widetilde{\varepsilon_2}(n)\widetilde{g_2}(n)\widetilde{\gamma_2}(j)\Gamma_1)-F_1(\widetilde{\varepsilon_2}(n_{j,k})\widetilde{g_2}(n)\widetilde{\gamma_2}(j)\Gamma_1)\right|,\] 
then as
\begin{align*}
\widetilde{\varepsilon_2}&(n_{j,k})\widetilde{g_2}(n)\widetilde{\gamma_2}(j)\\&=\begin{pmatrix}
1 & \frac{a(j)-k_2R(n)}{k_1}+t(n_{j,k})  & t(n_{j,k})R(n)-t(n)R(n)+\frac{t_U' P(n_{j,k})}{q_U}+\frac{a_U P(j)-k_2r_U P(n)}{k_1 q_U}\\
0 & 1 & R(n)\\
0 & 0 & 1
\end{pmatrix},
\end{align*}
we find, with \(P(n)=\sum_{i\geq 0}p_i\binom{n}{i},\) that a similar calculation to before gives
\begin{align*}
\widetilde{E}&\ll  \left|-t(n)\lfloor R(n) \rfloor +t(n_{j,k})\lfloor R(n)\rfloor-t(n_{j,k})R(n)+t(n)R(n)+\frac{t_U'(P(n)- P(n_{j,k}))}{q_U} \right|
\\ &\ll |\lfloor R(n) \rfloor -R(n)||t(n_{j,k})-t(n)|+|t_U'q_U^{-1}||P(n)-P(n_{j,k})|
\\& \leq \sum_{i\geq 1}(|t_i'|+|t_U' q_U^{-1} p_i|)\left|\binom{n_{j,k}}{i}-\binom{n}{i}\right|
\\ & \ll  \sum_{i\geq 1}(|t_i'|+|t_U'q_U^{-1} p_i|)\left|n_{j,k}-n\right|X^{i-1}
\\& \ll \delta_1^{-O(1)}\delta_2.
\end{align*}
In the last inequality we  used that \(\deg P \leq U.\) We thus establish a suitable variant of Lemma \ref{totalequidg2}.

Next, the analogue of Lemma \ref{delta2equid} gives \[\Vert \eta \circ \widetilde{g_2}\Vert_{C^\infty[X]} \ll \delta_2^{-O(1)}\] where
\[\eta(\widetilde{g_2}(n))=\frac{k_3}{k_1}R(n)+k_4\left(\frac{-k_2r_U}{k_1q_U}P(n)-t(n)R(n)+\frac{k_1 k_2}{2}\left(\frac{R(n)}{k_1}\right)^2\right)\]
for integers \(k_3,k_4\ll \delta_2^{-O(1)}\) and, as before, \(k_4\neq 0\) (by considering the specific \(r_i\) that we used to deduce   \eqref{thetamulti}). Write \(V=\deg R.\) Then the above smoothness condition and the fact that \(\deg P\leq V\) together imply that for all \(i> V\) we have \begin{equation}\label{tstar}\left\Vert k_4\left(-t_i^*+\frac{k_2}{2k_1}r_i^*\right)\right\Vert_{\mathbb{R}/\mathbb{Z}}\ll \frac{\delta_2^{-O(1)}}{X^i},\end{equation} where  \(t_i^*\) and \(r_i^*\) respectively represent the coefficients from
\[t(n)R(n)=\left(\sum_{i\geq 1}t_i'\binom{n}{i}\right)\left(\sum_{1\leq i\leq V}r_i\binom{n}{i}\right)=\sum_{ i\geq 1}t_{i}^*\binom{n}{i}\] and
\[R(n)^2=\left(\sum_{1\leq i\leq V}r_i\binom{n}{i}\right)^2=\sum_{ 1\leq i\leq 2V}r_{i}^*\binom{n}{i}.\]

In order to make use of \eqref{tstar}, we prove the following lemma.
\begin{lemma}[Irrational coefficients of high degree]\label{squarecoeff}  Suppose \(R,R^2\in \mathbb{R}[n]\setminus \mathbb{Q}[n]\) with \(R(0)=0\) and \(V=\deg R\geq 1.\)
 Then for some \(k>V\) we have \(r_k^*\in \mathbb{R}\setminus \mathbb{Q}.\) 
\end{lemma}
\begin{proof}
It is convenient to write \[R(n)^2=\left(\sum_{1\leq i\leq V}\rho_in^i\right)^2=\sum_{1\leq  i\leq 2V}\rho_i^*n^i.\]
Observe that if \(\rho_i^*\) is irrational then \(r_j^*\) is irrational for some \(j\geq i.\) Thus it suffices to show \(\rho_i^*\) is irrational for some \(i>V.\) We now have the formula
\[\rho_{V+j}^*=\rho_V\rho_j+\rho_{V-1}\rho_{j+1}+\dots +\rho_{j}\rho_{V}.\]
Let \(j_0=\max\{i:\rho_i\in\mathbb{R}\setminus \mathbb{Q}\},\) which is well-defined by the assumed property of \(R.\) 

If \(j_0<V\) then \(\rho_{V+j_0}^*\) is irrational by the above formula. Suppose instead that \(j_0=V,\) so that \(\rho_V\) is irrational.  We have 
\[\rho_{2V}^*=\rho_V^2.\] If \(\rho_V^2\) is irrational, then we are done. If \(\rho_V^2\) is rational, then from
\[\rho_{2V-1}^*=\rho_{V}\rho_{V-1}+\rho_{V-1}\rho_V,\] it follows that \(\rho_{2V-1}^*\) is irrational and we are done, or \(\rho_{V-1}=Q_{V-1}\rho_V\) for some \(Q_{V-1}\in \mathbb{Q}.\) By applying the same argument recursively, we see that at least one of \(\rho_{2V}^*,\dots,\rho_{V+1}^*\) is  irrational, or we have \(\rho_{i}=Q_i\rho_V,\) where \(Q_i\in \mathbb{Q}\) for all \(1\leq i\leq V\) (which covers all the coefficients of \(R\) because \(R(0)=0\)). However, the latter scenario implies \[R(n)^2=\rho_V^2\left(\sum_{1\leq i\leq V}Q_in^i\right)^2 \in \mathbb{Q}[n],\] contradicting the assumed property of \(R^2.\) 
\end{proof}
By Lemma \ref{squarecoeff},  there exists some irrational   coefficient \(r_{k}^*\) of \(R^2,\) with \(k>V.\)  For the largest such \(k,\) we see that \(r_k^*\) is also non-Liouville  with irrationality exponent at most \(\mu'\) (by considering the coefficient of \(n^k\) similarly to how we deduced  \eqref{thetamulti}).
 Since all \(t_i'\ll \delta_1^{-O(1)}X^{-1},\) and \(X\) is large in terms of the coefficients \(r_i,\) it follows that all \(t_i^*\ll\delta_1^{-O(1)}X^{-1}.\) Furthermore, from \eqref{tstar} and the triangle inequality, we have \[\left\Vert \frac{k_2 k_4}{2k_1}r_k^*\right\Vert_{\mathbb{R}/\mathbb{Z}}\ll\frac{\delta_2^{-O(1)}}{X^k}+\left\Vert k_4 t_k^*\right\Vert_{\mathbb{R}/\mathbb{Z}}=\frac{\delta_2^{-O(1)}}{X^k}+| k_4 t_k^*|\ll \frac{\delta_2^{-O(1)}}{X^k}+\frac{\delta_1^{-O(1)}\delta_2^{-O(1)}}{X},\] which contradicts the Diophantine  conditions on \(r_k^*\) as before, unless \(k_2=0.\) Therefore \(k_2=0.\) We then deduce from \eqref{tstar} that \(t_i^*\ll \delta_2^{-O(1)}X^{-i}\) for all \(i>V.\) By  considering the coefficient of largest degree in \(t(n)R(n)\) as the base case  (note that \eqref{thetamulti} implies \(t_i'=0\) when \(i>U\)) and working inductively, we obtain the estimate \(t_i'\ll\delta_2^{-O(1)}X^{-V-i}\) for all \(i\geq 1.\)

\subsection{Equidistribution in \texorpdfstring{$G_3/\Gamma_3$}{G3/Γ3}, II}\label{gamma3ii} We have thus far shown that our exponential sum is small unless
 \begin{equation}\label{majorarcsbracket}\theta=\frac{a}{q}+t\end{equation} for some  integers \(a,q\) with \(1\leq q\ll\delta_1^{-O(1)}\) and real \(t\ll \delta_2^{-O(1)}X^{-U-V},\) where \(U=\deg Q\) and \(V=\deg R.\)  
 
\begin{lemma}[Major and minor arc estimates for bracket polynomials, I]\label{majorbracketlemma} Let \(P,Q\in \mathbb{Q}[n]\) be such that \(P,Q:\mathbb{Z}\rightarrow \mathbb{Z}.\) Suppose \(R,R^2\in \mathbb{L}'[n]\setminus \mathbb{Q}[n].\) Assume \(\deg Q, \deg R \geq \max\{1,\deg P\}\) and \(P(0)=Q(0)=R(0)=0.\) Suppose  \(h\in\mathbb{Z}\) is positive such that \(hP,hQ\in\mathbb{Z}[n].\) 
Let \(\delta_1,\delta_2,\delta_3\) satisfy \eqref{deltas2} and let \(X\) be sufficiently large in terms of \(P,Q,\) and \(R.\) Assume \(\theta\in \mathbb{R}.\) Then 
\[\mathbb{E}_{n\in[X]} e(\theta(P(n)- Q(n)\lfloor R(n)\rfloor))\ll \delta_1^{1/2},\] 
or  \(\theta\) is of the form \eqref{majorarcsbracket}. Furthermore, for all \(\theta\) as in \eqref{majorarcsbracket} we have 
\begin{align*}&\mathbb{E}_{n\in[X]}e(\theta (P(n)-Q(n)\lfloor R(n)\rfloor))\\&=\mathbb{E}_{\substack{j~(\textup{mod } hq)\\ m~(\textup{mod } q)}}e\left(\frac{a}{q}\left(P(j)-Q(j)m\right)\right)\frac{1}{X}\int_0^Xe(tP(v)-tQ(v)R(v))~\mathrm{d}v +O(\delta_2^{-O(1)}\delta_3^{1/2}).\end{align*}
\end{lemma} 
 \begin{proof}
By the work in Section \ref{bracketpolysection}, we assume \(\theta\) is of the form \eqref{majorarcsbracket} and factorise
\begin{align*}\widetilde{g_1}(n)&=\begin{pmatrix}
1 & \theta Q(n) & \theta P(n)\\
0 & 1 & R(n)\\
0 & 0 & 1
\end{pmatrix}
\\&=\begin{pmatrix}
1 &  \left(\frac{a}{q}+t\right)Q(n) & \left(\frac{a}{q}+t\right)P(n)\\
0 & 1 &  R(n)\\
0 & 0 & 1
\end{pmatrix}
\\&=\begin{pmatrix}
1 & tQ(n) & tP(n)- tQ(n)R(n)\\
0 & 1 & 0\\
0 & 0 & 1
\end{pmatrix}\begin{pmatrix}
1 & 0 & 0\\
0 & 1 & R(n)\\
0 & 0 & 1
\end{pmatrix}\begin{pmatrix}
1 & \frac{a}{q}Q(n) & \frac{a}{q}P(n)\\
0 & 1 & 0\\
0 & 0 & 1
\end{pmatrix}
\\&=\widetilde{\varepsilon_3}(n)\widetilde{g_3}(n)\widetilde{\gamma_3}(n).
\end{align*}
Define \(G_3,\Gamma_3\) and Mal'cev basis element \(Z_1\) as before (with parameter \(q\)).
One can check that \(\widetilde{\gamma_3}(j)^{-1}\Gamma_3\widetilde{\gamma_3}(j)\subseteq \Gamma_1\) still holds, using the fact that \(Q:\mathbb{Z}\rightarrow \mathbb{Z}.\) In Mal'cev coordinates \[F_3(e^{t_1Z_1}\Gamma_3)=F_1(\widetilde{\varepsilon_3}(n_{j,k})e^{t_1Z_1}\widetilde{\gamma_3}(j)\Gamma_1)\] is equal to  \begin{align*}&e\left(tQ(n_{j,k})qt_1+tP(n_{j,k})-tQ(n_{j,k})R(n_{j,k})+\frac{a}{q}P(j)-\left(\frac{a}{q}Q(j)+tQ(n_{j,k})\right)\lfloor qt_1\rfloor\right)
\\&=e\left(tP(n_{j,k})-tQ(n_{j,k})R(n_{j,k})+\frac{a}{q}P(j)-\frac{a}{q}Q(j)\lfloor q t_1\rfloor\right)+O\left(\delta_1^{-O(1)}\delta_2^{-O(1)}X^{-V}|t_1|\right).\end{align*}
The error term from  Taylor expanding is small enough to ignore, as we take \(0\leq t_1< 1.\)

As before, \((\widetilde{g_3}(n)\Gamma_3)_{n\in [X]}\) is totally \(\delta_3\)-equidistributed in \(G_3/\Gamma_3.\)
 Let \(q'\) be such that \(\widetilde{\gamma_3}(n)\Gamma_3\) is `periodic' with period \(q'\geq 1\) (see earlier period calculation). We may choose \(h\) minimal (so that \(X\) is large in terms of \(h\)) since  any positive integer multiple of this choice gives the same main term as in the lemma. Take \(q'=hq\) and take \(P_{j,k}\) to be short arithmetic progressions of length \(\lceil \delta_3 X\rceil\) as before, but with common difference \(q'\) and whose elements are all congruent to \(j\) modulo \(q'.\) 
One can bound the quantity
\[\widetilde{E'}=\max_{j,k}\left|\mathbb{E}_{n\in P_{j,k}}F_1(\widetilde{\varepsilon_3}(n)\widetilde{g_3}(n)\widetilde{\gamma_3}(j)\Gamma_1)-F_1(\widetilde{\varepsilon_3}(n_{j,k})\widetilde{g_3}(n)\widetilde{\gamma_3}(j)\Gamma_1)\right|,\] 
 ignoring the error term discussed above, by
\begin{align*}
&\ll \left|e\left(t(P(n)-Q(n)R(n))-t(P(n_{j,k})-Q(n_{j,k})R(n_{j,k}))\right)-1\right|
\\&\ll |t|\left|P(n)-Q(n)R(n)-P(n_{j,k})-Q(n_{j,k})R(n_{j,k})\right|
\\& \ll|t| |n-n_{j,k}|\left(X^{(\deg P)-1}+X^{U+V-1}\right)
\\&  \ll \delta_2^{-O(1)}X^{-U-V} \delta_3 X X^{U+V-1}\ll \delta_{2}^{-O(1)}\delta_3,
\end{align*}
where we have used that \(\deg P \leq U+V.\)
We integrate as before: \begin{align*}\int_{G_3/\Gamma_3}F_3&=\int_0^1  e\left(tP(n_{j,k})-tQ(n_{j,k})R(n_{j,k})+\frac{a}{q}P(j)-\frac{a}{q}Q(j)\lfloor q t_1\rfloor\right) ~\mathrm{d}t_1 
\\&= e\left(tP(n_{j,k})-tQ(n_{j,k})R(n_{j,k})\right)\mathbb{E}_{m~(\textrm{mod } q)} e\left(\frac{a}{q}(P(j)-Q(j))m)\right).\end{align*}

Taking the average over \(j\) and \(k\) as in Section \ref{sectiong3}, we similarly arrive at
\begin{align*}&\mathbb{E}_{n\in[X]}e(\theta (P(n)-Q(n)\lfloor R(n)\rfloor))\\&=\mathbb{E}_{\substack{j~(\textrm{mod } q')\\ m~(\textrm{mod } q)}}e\left(\frac{a}{q}\left(P(j)-Q(j)m\right)\right)\frac{1}{X}\int_0^Xe(tP(v)-tQ(v)R(v))~\mathrm{d}v +O(\delta_2^{-O(1)}\delta_3^{1/2})\end{align*}
as desired.
\end{proof}

 Although it is natural to include the \(e(tP(v))\) term in the above formula, in our case it makes a negligible  contribution by Taylor's theorem because  \(\deg P\) is small. 

\begin{lemma}[Major and minor arc estimates for bracket polynomials, II]\label{majorbracketlemmaii}
The result of Lemma \ref{majorbracketlemma} holds when \(R^2\in \mathbb{Q}[n]\) instead of \(R^2\in \mathbb{L}'[n]\setminus \mathbb{Q}[n],\) provided that \(P,Q,\) and \(R\) are not all constant multiples of one another, and all other conditions remain the same.
\end{lemma}
\begin{proof}
Starting with the leading coefficient and working backwards (as in  the proof of Lemma \ref{squarecoeff}) we find that \(R/\beta\in\mathbb{Q}[n]\) for some \(\beta\in\mathbb{R}\setminus \mathbb{Q}\) such that \(\beta^2\in \mathbb{Q}.\) Note that such a \(\beta\) is non-Liouville with irrationality exponent \(2.\)

We modify the proof of Lemma \ref{majorarcsbracket}. By \eqref{thetamulti}, when \(q_i\neq 0\) we may write \[\theta=\frac{a_U-k_2r_U}{k_1q_U}+\frac{t_U'}{q_U}=\frac{a_i-k_2r_i}{k_1q_i}+\frac{t_i'}{q_i}.\]
Suppose \(k_2\neq 0.\) Since all \(r_i\) are rational multiples of \(\beta,\)  we find that  \(r_i/q_i=r_U/q_U\) is constant.  On the other hand, if \(q_i= 0\)  then by \eqref{thetamulti} we similarly deduce that   \(r_i=0.\) Thus, upon scaling \(\beta\) by a rational number if necessary, we deduce that \(R=\beta Q.\) 

Adapting the analogue of Lemma \ref{delta2equid} given in Section \ref{gamma2ii}, we may remove the contribution from \(r_i^*\) in \eqref{tstar} because \(r_i^*\in \mathbb{Q}\) and \(\Vert \alpha\Vert_{\mathbb{R}/\mathbb{Z}}/q\leq\Vert \alpha/q\Vert_{\mathbb{R}/\mathbb{Z}}\) for integers \(q\geq 1,\) and proceed similarly to before.  Moreover, recall that
\[\left\Vert\frac{k_3}{k_1}r_i+k_4\left(\frac{-k_2r_U}{k_1q_U}p_i-t_i^*+\frac{k_2}{2k_1}r_i^* \right)\right\Vert_{\mathbb{R}/\mathbb{Z}}\ll\frac{\delta_2^{-O(1)}}{X^i}.\]
Again \(r_i^*\in\mathbb{Q}\)  can effectively be ignored, and \(t_i^*\) is small and so is easily dealt with using the triangle inequality.  Since \(r_U/q_U=\beta\) and all \(r_i\) are rational multiples of \(\beta,\) we  arrive at \(P=k_3 R/(k_2k_4 \beta).\)

In other words, if  \(R=\beta Q\) holds for no \(\beta\in \mathbb{R}\) or if \(P= \beta' R\) holds for no \(\beta'\in \mathbb{R},\) then \(k_2=0\) and so our major arcs are of the claimed shape. 
The remainder of the argument is identical.
\end{proof}
The remaining case (which we include for interest only) is when \(R^2\in \mathbb{Q}[n],\)   \(R=\beta Q,\) and \(P=\beta' R\) for some \(\beta,\beta'\in \mathbb{R}.\) In this scenario  we find that we have major arcs at
\begin{equation}\label{majorarcsbracket2}\theta=\frac{a+b\beta}{q}+t\end{equation} for some integers \(a,b,q\) with \(b, q\ll\delta_1^{-O(1)},\) \(q\geq 1,\) and real \(t\ll \delta_2^{-O(1)}X^{-U-V},\)  where \(U=\deg Q\) and \(V=\deg R,\) and  the major arc estimate is the following.
\begin{lemma}[Major and minor arc estimates for bracket polynomials,  III] Let \(P,Q\in \mathbb{Q}[n]\) be such that \(P,Q:\mathbb{Z}\rightarrow \mathbb{Z}.\) Assume \(\deg Q\geq 1\) and \(Q(0)=0.\) Suppose \(\beta\in \mathbb{R}\setminus \mathbb{Q}\) satisfies \(\beta^2\in \mathbb{Q},\) and let \(R=\beta Q.\) Suppose \(f,g,h\in \mathbb{Z}\) with \(g,h\geq 1\) are such that \(P=fQ/g\) and \(hP,hQ,hR^2/2\in\mathbb{Z}[n].\)
Let \(\delta_1,\delta_2,\delta_3\) satisfy \eqref{deltas2} and let \(X\) be sufficiently large in terms of \(P,Q,\) and \(R.\) Assume \(\theta\in \mathbb{R}.\) Then 
\[\mathbb{E}_{n\in[X]} e(\theta(P(n)- Q(n)\lfloor R(n)\rfloor))\ll \delta_1^{1/2},\] or  \(\theta\) is of the form \eqref{majorarcsbracket2}.  Furthermore, for all \(\theta\) as in \eqref{majorarcsbracket2} we have 
\begin{multline*}
\mathbb{E}_{n\in[X]} e(\theta(P(n)- Q(n)\lfloor R(n)\rfloor))=\\ \mathbb{E}_{\substack{j~(\textup{mod } hq)\\ m~(\textup{mod } 2gq)}}e\left(\frac{a}{q}P(j)-\frac{a}{q}Q(j)m-\frac{b}{2q}R(j)^2-\frac{b}{2q}\left(m-\frac{f}{g}\right)^2\right)\\ \times \int_0^1 e\left(\frac{b}{2q}\left(x+\frac{f}{g}\right)^2\right)~\mathrm{d}x  \frac{1}{X}\int_{0}^{X}e( t(P(v)-Q(v)R(v)) ) ~\mathrm{d}v+O(\delta_2^{-O(1)}\delta_3^{1/2}).
\end{multline*}
\end{lemma}
\begin{proof} We adapt the proofs of Lemma \ref{majorbracketlemma} and Lemma \ref{majorbracketlemmaii} as follows. As in the proof of Lemma \ref{majorbracketlemmaii}, we deduce that major arcs take the form of \eqref{majorarcsbracket2}.  We factorise
\begin{align*}\widetilde{g_1}(n)
&=\begin{pmatrix}
1 & tQ(n) & tP(n)- tQ(n)R(n)\\
0 & 1 & 0\\
0 & 0 & 1
\end{pmatrix}\begin{pmatrix}
1 & \frac{b\beta}{q}Q(n) & \frac{b}{2q}R(n)^2+\frac{b\beta}{q}P(n)\\
0 & 1 & R(n)\\
0 & 0 & 1
\end{pmatrix}\\ &\quad \times\begin{pmatrix}
1 & \frac{a}{q}Q(n) & \frac{a}{q}P(n)-\frac{b}{2q}R(n)^2\\
0 & 1 & 0\\
0 & 0 & 1
\end{pmatrix}
\\&=\widetilde{\varepsilon_3}(n)\widetilde{g_3}(n)\widetilde{\gamma_3}(n).
\end{align*}
Consider  \(G_3=\left\{\begin{psmallmatrix}
1 & bgx & \frac{bg^2qx^2}{2}+bfx\\
0 & 1 & gqx\\
0 & 0 & 1
\end{psmallmatrix}: x\in \mathbb{R}\right\}\) and  \(\Gamma_3=\left\{\begin{psmallmatrix}
1 & bgx & \frac{bg^2qx^2}{2}+bfx\\
0 & 1 & gqx\\
0 & 0 & 1
\end{psmallmatrix}: x\in\mathbb{Z}\right\},\)
with Mal'cev basis given by
\(Z_1=\begin{psmallmatrix}
0 & bg & bf\\
0 & 0 & gq\\
0 & 0 & 0
\end{psmallmatrix}.\)   We may choose \(g\) and \(h\) minimal (so that \(X\) is large in terms of \(g\) and \(h\)) since, as before,  any positive integer multiples would give the same main term as in the lemma.
The computation of the major arc estimate is then essentially a combination of those given in Lemma \ref{majorarcrootestimate} and Lemma \ref{majorbracketlemma}.
\end{proof}
 
\subsection{Positivity at major arcs, II}
Consider the major arc approximation from Lemma \ref{majorbracketlemma} and Lemma \ref{majorbracketlemmaii}. 
  On evaluating the average over \(j\) and \(m,\) we have \[\mathbb{E}_je\left(\frac{a}{q}P(j)\right)\mathbb{E}_m e\left(-\frac{a}{q}Q(j)m\right)=\mathbb{E}_je\left(\frac{a}{q}P(j)\right)1_{aQ(j)\in q\mathbb{Z}}=\mathbb{E}_{j}1_{aQ(j)\in q\mathbb{Z}}\geq \frac{1}{hq},\] where the first equality comes from summing a geometric series and the fact that \(Q:\mathbb{Z}\rightarrow \mathbb{Z},\)  where the second equality is deduced from our assumption that \(P=-\ell Q\) for some \(\ell \in \mathbb{Z},\) and where the inequality holds because \(Q(0)=0.\)

  For the integral, note that \(e(tP(v)-tQ(v)R(v))=e(sv^{U+V})+O(\delta_2^{-O(1)}X^{-1})\) for \(s=-tq_Ur_V.\) By a similar analysis to before, for all integers  \(d\geq 2,\) we have 
 \[\Re\frac{1}{X}\int_0^Xe(sv^d)~\mathrm{d}v=\frac{1}{X}\int_0^X\cos(2\pi |s|v^d)~\mathrm{d}v\gg \delta_2^{O(1)},\] where the only new information needed is that
 \[\int_0^{(1/4)^{1/d}} \cos(2\pi w^d)~\mathrm{d}w \geq \int_0^{\sqrt{1/4}} \cos(2\pi w^2)~\mathrm{d}w \geq  0.38 \] and \[ \left| \int_{(1/4)^{1/d}}^{(3/4)^{1/d}} \cos(2\pi w^2)~\mathrm{d}w\right| \leq \left(\frac{3}{4}\right)^{1/d}-\left(\frac{1}{4}\right)^{1/d}\leq \sqrt{\frac{3}{4}}-\sqrt{\frac{1}{4}}\leq 0.37.\]
 
 Hence  there is a small absolute constant \(c'>0\) depending only on \(U\) and \(V\)   such that   for all  \(X\)  large enough in terms of \(Q,R,\) and \(\ell,\) we have
 \[|\mathbb{E}_{n\in[X]}e(-\theta Q(n)\lfloor R(n)+\ell\rfloor)|\leq X^{-c'/\mu'},\] or  \(\theta\) is of the form \eqref{majorarcsbracket} which implies
\[ \Re \mathbb{E}_{n\in[X]}e(-\theta Q(n)\lfloor R(n)+\ell\rfloor)\geq X^{-c'/\mu'}.\]  From this we deduce (a slightly more precise version of) Theorem \ref{polythm}.
    
\bibliography{main}
\bibliographystyle{amsalpha}

\end{document}